\documentclass[12pt]{article}
\usepackage{amsmath,amssymb,amsthm}
\usepackage[frenchb,english]{babel}
\usepackage[latin1]{inputenc}
\usepackage{graphicx}
\usepackage{amsfonts}
\usepackage{amssymb}
\usepackage{latexsym}
\usepackage{amscd}
\usepackage{color}
\usepackage{shadow}
\usepackage{a4wide}
\numberwithin{equation}{section}

\theoremstyle{plain}
\newtheorem{thm}{Theorem}[section]
\newtheorem{prop}[thm]{Proposition}
\newtheorem{cor}[thm]{Corollary}
\newtheorem{lemma}[thm]{Lemma}
\newtheorem{dfn}[thm]{Definition}
\theoremstyle{remark}
\newtheorem{rmq}[thm]{Remark}

\begin{document}

\title{
Precise smoothing effect in the exterior of balls}

\author{Oana Ivanovici  \\Universite Paris-Sud, Orsay,\\
Mathematiques, Bat. 430, 91405 Orsay Cedex, France\\
oana.ivanovici@math.u-psud.fr}

\date{ }

\bigskip
\bigskip
\bigskip
\bigskip
\maketitle

\bigskip
\bigskip

\section{Introduction}
We are interested in this article in investigating the smoothing
effect properties of the solutions of the Schr\"odinger equation.
Since the work by Craig, Kapeller and Strausss~\cite{cks95}, Kato~\cite{Ka}, Constantin and Saut~\cite{CoSa}
establishing the smoothing property, many works have dealt with
the understanding of this effect. In particular the work by
Doi~\cite{Do} and Burq~\cite{Bu1,Bu2} shows that it is closely
related to the infinite speed of propagation
 for the solutions of Schr\"odinger equation. Roughly speaking, if one considers a wave
  packet with wave length $\lambda$, it is known that it propagates with speed $\lambda$
  and  the wave will stay in any bounded domain only for a time of order $1/\lambda$.
  As a consequence, taking the $L^2$ in time norm will lead to an improvement of
  $1/\lambda^{1/2}$ with respect to taking an $L^\infty$ norm, leading to a gain
  of $1/2$ derivatives. This heuristic argument can be transformed into a proof of
  the smoothing effect either by direct calculations (in the case of the free Schr\"odinger
  equation) or by means of resolvent estimates (see \cite{BeKl} for the case of a perturbation
  by a potential or \cite{bgt03} for the boundary value problem). In view of this simple heuristics,
   it is natural to ask whether one can refine (and improve) such smoothing type estimates if
   one considers smaller space domains (whose size will shrink as the wave length increases).
   A very natural context in which one can test this heuristics is the case of the exterior
of a convex body (or more generally the exterior of several convex
bodies), in which case natural candidates for the $\lambda$
dependent domains are $\lambda^{-\alpha}$ neighborhoods of the
boundary. This is the main aim of this paper. To keep the paper at
a rather basic technical level, we choose to consider only balls,
for which direct calculations (with Bessel functions) can be
performed.  Our first result reads as follows:
\begin{thm}\label{thm1}
Let $\Omega=\mathbb{R}^{3}\setminus B(0,1)$, $T>0$,
$0\leq\alpha<\frac{2}{3}$ and $\lambda\geq 1$. Let $\psi$ and
$\chi\in C^{\infty}_{0}(\mathbb{R}^{*})$ be smooth functions with compact support, $\psi=1$ near $1$, $\chi=1$ near $0$. Set
$\chi_{\lambda}(|x|):=\chi(\lambda^{\alpha}(|x|-1))$, where
$x$ denotes the variable on $\Omega$. Then one
has
\begin{itemize}
\item For $s\in [-1,1]$ and
$v(t)=\int_{0}^{t}e^{i(t-\tau)\Delta_{D}}\psi(\frac{-\Delta_{D}}{\lambda^{2}})\chi_{\lambda}gd\tau$
\begin{equation}\label{i}
\|\chi_{\lambda}\psi(\frac{-\Delta_{D}}{\lambda^{2}})v\|_{L_{T}^{2}H_{D}^{s+1}(\Omega)}\leq C
\lambda^{-\frac{\alpha}{2}}\|\psi(\frac{-\Delta_{D}}{\lambda^{2}})\chi_{\lambda}g\|_{L_{T}^{2}H_{D}^{s}(\Omega)},
\end{equation}

\item For $s\in[0,1]$
\begin{equation}\label{ii}
\|\chi_{\lambda}e^{it\Delta_{D}}\psi(\frac{-\Delta_{D}}{\lambda^{2}})u_{0}\|_{L^{2}_{T}H_{D}^{s+\frac{1}{2}}(\Omega)}\leq C
\lambda^{s-\frac{\alpha}{4}}\|\psi(\frac{-\Delta_{D}}{\lambda^{2}})u_{0}\|_{L^{2}(\Omega)}.
\end{equation}
\end{itemize}
Here the constants $C$ do not depend on $T$, i.e. the estimates are \emph{global} in time.
\end{thm}
Using this result, we can deduce new Strichartz type estimates for
the solution of the linear Schr\"odinger equation in the exterior,
$\Omega$, of a smooth bounded obstacle
$\Theta\subset\mathbb{R}^{3}$,
\begin{equation}\label{LS}
(i\partial_{t}+\Delta_{D,N})u =0,\\ \qquad u(0,x)=u_{0}(x), \qquad
u|_{\partial\Omega}=0 \text{ or }\partial_n u
\mid_{\partial\Omega}=0,
\end{equation}
that we denote respectively by $e^{it\Delta_D}u_0$ and $e^{it
\Delta_N}u_0$.
\begin{dfn}
Let $q,r\geq
2$, $(q,r,d)\neq(2,\infty,2)$, $2\leq p\leq\infty$. A pair $(q,r)$ is called admissible in dimension $d$ if $q$, $r$ satisfy the scaling admissible condition
\begin{equation}\label{admis}
\frac{2}{p}+\frac{d}{q}=\frac{d}{2}.
\end{equation} 
\end{dfn}
\begin{thm}\label{thm2}(Strichartz estimates)
Let $\Theta=B(0,1)\subset\mathbb{R}^{3}$ and
$\Omega=\mathbb{R}^{3}\setminus \Theta$, $T>0$ and $(p,q)$ an admissible pair in dimension $3$. Let $\epsilon>0$ be an arbitrarily small constant. Then there exists a constant $C>0$ such that, for all
$u_{0}\in H_{D,N}^{\frac{4}{5p}+\epsilon}(\Omega)$ the following holds
\begin{equation}\label{stri}
\|e^{it\Delta_{D,N}}u_{0}\|_{L^{p}([-T,T],L^{q}(\Omega))}\leq
C\|u_{0}\|_{H_{D,N}^{\frac{4}{5p}+\epsilon}(\Omega)}.
\end{equation}
Moreover, a similar result holds true for a class of trapping
obstacles (Ikawa's example), i.e. for the case where $\Theta$ is a finite union
of balls in $\mathbb{R}^{3}$.
\end{thm}
As a consequence of Theorem \ref{thm2} we deduce new global well-posedness results for
the non-linear Schr\"odinger equation in the exterior of several
convex obstacles, improving previous results by Burq, Gerard and
Tzvetkov~\cite{bgt03}. Consider the nonlinear Schr\"{o}dinger
equation on $\Omega$ subject to Dirichlet boundary condition
\begin{equation}\label{nls}
(i\partial_{t}+\Delta_{D})u=F(u)\quad \text{in} \quad
\mathbb{R}\times\Omega,\quad u(0,x)=u_{0}(x),\quad \text{on} \quad
\Omega,\quad u|_{\partial\Omega}=0.
\end{equation}
The nonlinear interaction $F$ is supposed to be of the form
$F=\partial V/\partial\bar{z}$, with $F(0)=0$, where the
"potential" $V$ is real valued and satisfies $V(|z|)=V(z)$,
$\forall z\in\mathbb{C}$. Moreover we suppose that $V$ is of class
$C^{3}$, $|D_{z,\bar{z}}^{k}V(z)|\leq C_{k}(1+|z|)^{4-k}$ for
$k\in\{0,1,2,3\}$ and that $V(z)\geq -(1+|z|)^{\beta}$, for some
$\beta<2+\frac{4}{d}$, $d=3$ (the last assumption avoid blow-up in the
focussing case).

Some phenomena in physics turn out to be modeled by exterior
problems and one may expect rich dynamics under various boundary
conditions. A first step in this direction is to establish well
defined dynamics in the natural spaces determined by the
conservation lows associated to \eqref{nls}. If $u(t,.)\in
H^{1}_{0}(\Omega)\bigcap H^{2}(\Omega)$ is a solution of
\eqref{nls} then it satisfies the conservation lows
\begin{equation}\label{conserv}
\frac{d}{dt}\int_{\Omega}|u(t,x)|^{2}dx=0; \quad
\frac{d}{dt}\Big(\int_{\Omega}|\nabla
u(t,x)|^{2}dx+\int_{\Omega}V(u(t,x))dx\Big)=0
\end{equation}
and therefore for a large class of potentials $V$ the quantity
$\|u(t,.)\|_{H^{1}_{0}(\Omega)}$ remains finite  along the
trajectory starting from $u_{0}\in H^{1}_{0}(\Omega)\bigcap
H^{2}(\Omega)$. This fact makes the study of \eqref{nls} in the
energy space $H^{1}_{0}(\Omega)$ of particular interest. It is
also of interest to study \eqref{nls} in $L^{2}(\Omega)$: the main
issue in the analysis is that the regularities of $H^{1}$ and
$L^{2}$ are a priori too poor to be achieved by the classical
methods for establishing local existence and uniqueness for
\eqref{nls}. We state the result concerning finite energy
solutions, which will be a consequence of Theorem \ref{thm2}.
\begin{thm}\label{thm3}(Global existence theorem)
Let $\Omega=\mathbb{R}^{3}\setminus B(0,1)$. For any $u_{0}\in
H_{0}^{1}(\Omega)$ the initial boundary value problem \eqref{nls}
has a unique global solution $u\in
C(\mathbb{R},H_{0}^{1}(\Omega))$ satisfying the conservation laws
\eqref{conserv}. Moreover, for any $T>0$ the flow map
$u_{0}\rightarrow u$ is Lipschitz continuous from any bounded set
of $H_{0}^{1}(\Omega)$ to $C(\mathbb{R},H_{0}^{1}(\Omega))$.
\end{thm}
\begin{rmq}
In \cite{bgt03}, Strichartz type estimates with loss of
$\frac{1}{p}$ derivative have been obtained for the Schrödinger
equation in the exterior of a non-trapping obstacle
$\Theta\subset\mathbb{R}^{d}$ which allowed to prove the same
global existence result in dimension $3$ provided $\|u_{0}\|_{H_{0}^{1}(\Omega)}$
is sufficiently small.
\end{rmq}
The Cauchy problem associated to \eqref{nls} has been extensively
studied in the case $\Omega=\mathbb{R}^{d}$, $d\geq 2$, by Bourgain
\cite{Bourg}, Cazenave \cite{Ca}, Sulem et Sulem \cite{Sul},
Ginibre and Velo \cite{Gi}, Kato \cite{Ka} and the theory of
existence of finite energy solutions to \eqref{nls} for potentials
$V$ with polynomial growth has been much developed. Roughly
speaking the argument for establishing finite energy solutions of
\eqref{nls} consists in combining $H^{1}$ local well-posedness
with conservation laws \eqref{conserv} which provide a control on
the $H^{1}$ norm.

The article is written as follows: in Section \ref{smef} we obtain bounds for the $L^{2}$ norms for the outgoing solution of the Helmholtz equation that will be used in Section \ref{secthm1} in order to prove Theorem \ref{thm1}. In Section \ref{away} we use a strategy inspired from \cite{stta02}, \cite{bgt03}
to handle the case when the solution is supported outside a neighborhood of
$\partial\Omega$; in Section \ref{secthm2} we achieve the proof of Theorem
\ref{thm2}. The last section is dedicated to the applications of these results; precisely, we give the proofs of
Theorems \ref{thm2} and \ref{thm3} in the case where the obstacle consists of a
union of balls. In the Appendix we recall some properties of the Hankel functions.

\textbf{Acknowledgments}: \emph{This result is part of the
author's PhD thesis in preparation at University Paris Sud, Orsay,
under Nicolas Burq's direction.}

\section{Precise smoothing effect}\label{smef}
\subsection{Preliminaries}
Let $\Omega\subset\mathbb{R}^{d}$, $d\geq2$ be a smooth domain .
For $s\geq0$, $p\in [1,\infty]$ we denote by $W^{s,p}(\Omega)$ the
Sobolev spaces on $\Omega$. We write $L^{p}$ and $H^{s}$ instead
of $W^{0,p}$ and $W^{s,2}$.
By $\Delta_{D}$ (resp. $\Delta=\Delta_{N}$) we denote the
Dirichlet Laplacian (resp. Neumann Laplacian) on $\Omega$, with
domain $H^{2}(\Omega)\cap H^{1}_{0}(\Omega)$ (resp. $\{u\in
H^{2}(\Omega)| \partial_{n}u|_{\partial\Omega}=0\}$). We next
define the dual space of $H^{1}_{0}(\Omega)$, $H^{-1}(\Omega)$,
which is a subspace of $D'(\Omega)$. We construct $H^{-s}(\Omega)$
via interpolation and due to \cite[Cor.4.5.2]{Be} we have the
duality between $H_{0}^{s}(\Omega)$ and $H^{-s}(\Omega)$ for $s\in
[0,1]$.
\begin{prop}\label{prop1}
\begin{itemize}
Let $d\geq2$ and $\Omega\subset\mathbb{R}^{d}$ be a smooth domain.
The following continuous embeddings hold:
\item  $H^{1}_{0}(\Omega)\subset L^{q}(\Omega)$, \quad $2\leq q\leq
\frac{2d}{d-2}$ ($p<\infty$ if $d=2$),
\item $H^{s}_{D}(\Omega)\subset L^{q}(\Omega)$, \quad
$\frac{1}{2}-\frac{1}{q}=\frac{s}{d}$, $s\in [0,1),$
\item$H_{D}^{s+1}(\Omega)\subset W^{1,q}(\Omega)$, \quad
$\frac{1}{2}-\frac{1}{q}=\frac{s}{d}$, $s\in [0,1)$,
\item$W^{s,p}(\Omega)\subset L^{\infty}(\Omega)$, \quad
$s>\frac{d}{p}$, $p\geq1$.
\end{itemize}
\end{prop}
\begin{proof}
The proof follows from the Sobolev embeddings on $\mathbb{R}^{d}$
and the use of extension operators.
\end{proof}
Let $\psi,\chi\in C^{\infty}_{0}(\mathbb{R}^{*})$ be smooth
functions such that $\chi=1$ in a neighborhood of $0$ and for all $\tau\in\mathbb{R}_{+}$, $\sum_{k\geq 0}\psi(2^{-2k}\tau)=1$. For
$\lambda>0$ let $\chi_{\lambda}(|x|):=\chi(\lambda^{\alpha}(|x|-1))$, where $x$ denotes variable on
$\Omega$. We introduce spectral localizations which commute with
the linear evolution. Since the spectrum of $-\Delta_{D}$ is confined
to the positive real axis it is convenient to introduce
$\lambda^{2}$ as a spectral parameter. We will consider the linear
problem \eqref{LS} with initial data localized at frequency $\lambda$,
$u|_{t=0}=\psi(-\frac{\Delta_{D}}{\lambda^{2}})u_{0}$.
\begin{rmq}
In order to prove Theorem \ref{thm2} it will be enough to prove \eqref{stri} with the initial data of the
form $\psi(-\frac{\Delta}{\lambda^{2}})u_{0}$, since than we have for
some $\epsilon$ arbitrarily small
\begin{equation}
\|e^{it\Delta_{D}}u_{0}\|_{L^{p}([-T,T],L^{q}(\Omega))}\thickapprox\|(e^{it\Delta_{D}}
\psi(-\frac{\Delta_{D}}{2^{2j}})u_{0})_{j}\|_
{L^{p}([-T,T],L^{q}(\Omega))l^{2}_{j}}\leq
\end{equation}
\[
\leq\|(e^{it\Delta_{D}}
\psi(-\frac{\Delta_{D}}{2^{2j}})u_{0})_{j}\|_
{l^{2}_{j}L^{p}([-T,T],L^{q}(\Omega))}\leq
\Big(\sum_{j}
2^{2j(\frac{4}{5p}+\epsilon)}\|\psi(-\frac{\Delta_{D}}{2^{2j}})u_{0}\|^{2}_{L^{2}
(\Omega)}\Big)^{\frac{1}{2}}\thickapprox
\|u_{0}\|_{H^{\frac{4}{5p}+\epsilon}_{D}(\Omega)}.
\]
\end{rmq}

\subsection{Estimates in a small neighborhood of the boundary}
In what follows let $d=3$. We study first the outgoing
solution to the equation
\begin{equation}\label{2d}
(\Delta_{D} +\lambda^{2})w=
\chi_{\lambda}f,\quad w|_{\partial\Omega}=0.
\end{equation}
In what follows we will establish high frequencies bounds for the
$L^{2}$ norm of $w$ in a small neighborhood
$\Omega_{\lambda}=\{|x|\lesssim\lambda^{-\alpha}\}$ of size $\lambda^{-\alpha}$ of the boundary. We notice that a
ray with transversal, equal-angle reflection spends in the
neighborhood $\Omega_{\lambda}$ a time $\simeq\lambda^{-\alpha}$. If the
ray is diffractive then the time spent in $\Omega_{\lambda}$
equals $\lambda^{-\frac{\alpha}{2}}$. 

We analyze the outgoing solution of \eqref{2d} outside the unit ball of $\mathbb{R}^{3}$.  The first step is to introduce polar coordinates and to write the expansion in spherical harmonics of the solution to 
\begin{equation}\label{helmhom}
\left\{
             \begin{array}{ll}
                (\Delta_{D}+\lambda^{2})\tilde{w}=0 \quad \text{on}\quad \Omega=\{x\in\mathbb{R}^{3}| |x|>1\},\\
               \tilde{w}|_{\partial\Omega}=\tilde{f}\quad \text{on} \quad\ \mathbb{S}^{2},\\
               r(\partial_{r}w-i\lambda w)\rightarrow_{r\rightarrow\infty}0.\\
                \end{array}
                \right.
\end{equation}
In this coordinates the Laplace operator on $\Omega$ writes
\begin{equation}
\Delta_{D}=\partial^{2}_{r}+\frac{2}{r}\partial_{r}+\frac{1}{r^{2}}\Delta_{\mathbb{S}^{2}},
\end{equation}
where $\Delta_{\mathbb{S}^{2}}$ is the Laplace operator on the sphere $\mathbb{S}^{2}$. Thus the solution $\tilde{w}$ of \eqref{helmhom} satisfies
\begin{equation}
r^{2}\partial^{2}_{r}\tilde{w}+2r\partial_{r}\tilde{w}+(\lambda^{2}r^{2}+\Delta_{\mathbb{S}^{2}})\tilde{w}=0,\quad r>1.
\end{equation}
In particular, if $\{e_{j}\}$ is an orthonormal basis of $L^{2}(\mathbb{S}^{2})$ consisting of eigenfunctions of $\Delta_{\mathbb{S}^{2}}$, with eigenvalues $-\mu^{2}_{j}$ and if $\omega$ denotes the variable on the sphere $\mathbb{S}^{2}$, we can write
\begin{equation}
\tilde{w}(r\omega)=\sum_{j}\tilde{w}_{j}(r)e_{j}(\omega),\quad r\geq 1,
\end{equation}
where the functions $\tilde{w}_{j}(r)$ satisfy
\begin{equation}
r^{2}\tilde{w}''_{j}(r)+2r\tilde{w}'_{j}(r)+(\lambda^{2}r^{2}-\mu^{2}_{j})\tilde{w}_{j}(r)=0,\quad r>1.
\end{equation}
This is a modified Bessel equation, and the solution satisfying the radiation condition $ r(\partial_{r}w-i\lambda w)\rightarrow_{r\rightarrow\infty}0$ is of the form
\begin{equation}\label{solhelmhom}
\tilde{w}_{j}(r)=a_{j}r^{-\frac{1}{2}}H_{\nu_{j}}(\lambda r),
\end{equation}
where $H_{\nu_{j}}(z)$ denote the Hankel function. Recall that the Hankel function is given by
\begin{equation}
H_{\nu}(z)=(\frac{2}{\pi z})^{1/2}\frac{e^{i(z-\pi\nu/2-\pi/4)}}{\Gamma(\nu+1/2)}\int_{0}^{\infty}e^{-s}s^{\nu-1/2}(1-\frac{s}{2iz})^{\nu-1/2}ds, 
\end{equation}
where
\[
\Gamma(\nu+1/2)=\int_{0}^{\infty}e^{-s}s^{\nu-1/2}ds,
\]
and it is valid for $Re\nu>\frac{1}{2}$ and $-\pi/2<\arg z<\pi$. Also, in \eqref{solhelmhom} $\nu_{j}$ is given by
\begin{equation}
\nu_{j}=(\mu^{2}_{j}+\frac{1}{4})^{1/2}
\end{equation}
and the coefficients $a_{j}$ are determined by the boundary condition $\tilde{w}_{j}(1)=<\tilde{f},e_{j}>$, so
\begin{equation}
a_{j}=\frac{<\tilde{f},e_{j}>}{H_{\nu_{j}}(\lambda)}.
\end{equation}
Let us introduce the self-adjoint operator
\begin{equation}
A=(-\Delta_{\mathbb{S}^{2}}+\frac{1}{4})^{1/2},\quad Ae_{j}=\nu_{j}e_{j}.
\end{equation}
Then the solution of \eqref{helmhom} writes, formally,
\begin{equation}\label{tildewro}
\tilde{w}(r\omega)=r^{-1/2}\frac{H_{A}(\lambda r)}{H_{A}(\lambda)}\tilde{f}(\omega), \quad \omega\in\mathbb{S}^{2}.
\end{equation}
\begin{prop}(see \cite[Chp.3]{tay96})
The spectrum of $A$ is $spec(A)=\{m+\frac{1}{2}|m\in\mathbb{N}\}$.
\end{prop}
\begin{rmq}\label{prophank}(see \cite[Chp.3]{tay96})
The Hankel function $H_{m+1/2}(z)$ and Bessel functions of order $m+\frac{1}{2}$ are all elementary functions of $z$. We have 
\begin{equation}
H_{m+1/2}(z)=(\frac{2z}{\pi})^{1/2}q_{m}(z),\quad q_{m}(z)=-i(-1)^{m}(\frac{1}{z}\frac{d}{dz})^{m}(\frac{e^{iz}}{z}).
\end{equation}
We deduce that
\begin{equation}
r^{-1/2}\frac{H_{m+1/2}(\lambda r)}{H_{m+1/2}(\lambda)}=r^{-m-1}e^{i\lambda(r-1)}\frac{p_{m}(\lambda r)}{p_{m}(\lambda)},\quad p_{m}(z)=i^{-m-1}\sum_{k=0}^{m}(i/2)^{k}\frac{(m+k)!}{k!(m-k)!}z^{m-k}.
\end{equation}
\end{rmq}
We now look for a solution to \eqref{2d}. If we write
\[
f(r\omega)=\sum_{j}f_{j}(r)e_{j}(\omega),\quad w(r\omega)=\sum_{j}w_{j}(r)e_{j}(\omega),
\]
then the functions $w_{j}(r)$ satisfy
\begin{equation}\label{wj}
r^{2}w''_{j}(r)+2rw'_{j}(r)+(\lambda^{2}r^{2}-\mu^{2}_{j})w_{j}(r)=r^{2}f_{j}(r),\quad r>1
\end{equation}
together with the vanishing condition at $r=1$ and Sommerfield radiation condition when $r\rightarrow\infty$.
Applying the variation of constants method together with the outgoing assumption and the formula
 \eqref{tildewro}, we obtain
\begin{equation}\label{solj}
w_{j}(r)=\int_{0}^{\infty}G_{\nu_{j}}(r,s,\lambda)\chi_{\lambda}(s)f_{j}(s)s^{2}ds,\quad \nu_{j}=(\mu^{2}_{j}+\frac{1}{4})^{1/2},
\end{equation}
where $G_{\nu}(r,s,\lambda)$ is the Green kernel for the differential operator
\begin{equation}
L_{\nu}=\frac{d^{2}}{dr^{2}}+\frac{2}{r}\frac{d}{dr}+(\lambda^{2}-\frac{\mu^{2}}{r^{2}}), \quad \nu=(\mu^{2}+\frac{1}{4})^{1/2}.
\end{equation}
The fact that $L_{\nu}$ is self-adjoint implies that $G_{\nu}(r,s,\lambda)=G_{\nu}(s,r,\lambda)$ and from the radiation condition we obtain (for $\lambda$ real)
\begin{equation}\label{helmhoms}
G_{\nu}(r,s,\lambda)=
\left\{
             \begin{array}{ll}
               \frac{\pi}{2i}(rs)^{-1/2}\Big(J_{\nu}(s\lambda)-\frac{J_{\nu}(\lambda)}{H_{\nu}(\lambda)}H_{\nu}(s\lambda)\Big)H_{\nu}(r\lambda),\quad r\geq s, \\
               \frac{\pi}{2i}(rs)^{-1/2}\Big(J_{\nu}(r\lambda)-\frac{J_{\nu}(\lambda)}{H_{\nu}(\lambda)}H_{\nu}(r\lambda)\Big)H_{\nu}(s\lambda),\quad r\leq s. \\
                \end{array}
                \right.
\end{equation}
In what follows we will look for estimates of the $L^{2}$ norm of
$w_{j}$ on the interval $[1,1+\lambda^{-\alpha}]$. This problem
has to be divided in several classes, according whether
$\nu_{j}/\lambda$ is less than, nearly equal or greater than $1$. We
distinguish also the simple case of determining bounds when the
argument is much larger than the order. 
Let us explain the meaning of this: in fact, applying the operator $L_{\nu_{j}}$ to $rw_{j}(r)$ instead of $w_{j}(r)$, we eliminate the term involving $w'_{j}(r)$ in \eqref{wj}. On the characteristic set we have
\begin{equation}\label{chv}
\rho^{2}_{j}+\frac{\mu^{2}_{j}}{r^{2}}=\lambda^{2}, \quad \nu_{j}=(\mu^{2}_{j}+\frac{1}{4})^{1/2},
\end{equation}
where $\rho_{j}$ denotes the dual variable of $r$. Let $\theta_{j}$ be defined by $\tan\theta_{j}=\frac{\rho_{j}}{\mu_{j}}$, then for $\lambda$ big enough and $r$ in a small neighborhood of $1$ we can estimate 
\[
\tan^{2}\theta_{j}+1\simeq\frac{\lambda^{2}}{\nu^{2}_{j}}.
\]
\begin{itemize}
\item When the quotient $\frac{\nu_{j}}{\lambda}$ is smaller than a constant $1-\epsilon_{0}$ where $\epsilon_{0}$ is fixed, strictly positive, this corresponds to an angle $\theta_{j}$ between some fixed direction $\theta_{0}$ and $\pi/2$, with $\tan\theta_{0}=\epsilon_{0}$, and thus to a ray hitting the obstacle transversally. In this case we show that, since a unit speed 
bicharacteristic spends in the $\lambda$-depending neighborhood $\Omega_{\lambda}$ a time $\lambda^{-\alpha}$, we have the following
\begin{prop}\label{propba}
Let $\epsilon_{0}>0$ be fixed, small. Then there exists a constant $C=C(\epsilon_{0})$ such that 
\[
\|w_{j}\|_{L^{2}([1,1+\lambda^{-\alpha}])}\leq C\lambda^{-(1+\alpha)}\|f_{j}\|_{L^{2}([1,1+\lambda^{-\alpha}])}
\] 
uniformly for $j$ such that $\{\nu_{j}/\lambda\leq 1-\epsilon_{0}\}$. 
\end{prop}

\item When the quotient $\frac{\nu_{j}}{\lambda}$ is close enough to $1$ the angles $\theta_{j}$ become very small and this is the case of a diffractive ray, which spends in $\Omega_{\lambda}$ a time proportional to $\lambda^{-\alpha/2}$. In this case $\tan\theta_{j}\simeq \sqrt{1-\frac{\nu^{2}_{j}}{\lambda^{2}}} $ and we show the following

\begin{prop}\label{propsa}
If $1-\frac{\nu_{j}}{\lambda}\simeq\lambda^{-\beta}$ for some $\beta>0$ then 
\[
\|w_{j}\|_{L^{2}([1,1+\lambda^{-\alpha}])}\leq C\lambda^{-1-\alpha/2-(\alpha-\beta)}\|f_{j}\|_{L^{2}([1,1+\lambda^{-\alpha}])}\ if \ \beta\leq \alpha
\]
and 
\[
\|w_{j}\|_{L^{2}([1,1+\lambda^{-\alpha}])}\leq C\lambda^{-(1+\alpha/2)}\|f_{j}\|_{L^{2}([1,1+\lambda^{-\alpha}])}\ if \ \beta\geq \alpha.
\]
\end{prop}
\item In the elliptic case $\frac{\lambda}{\nu_{j}}\ll1$ or $\frac{\lambda}{\nu_{j}}\in[\epsilon_{0},1-\epsilon_{0}]$ for some small $\epsilon_{0}>0$ there is nothing to do since away from the characteristic variety \eqref{chv} we have nice bounds of the solution of the solution $w_{j}(r)$ of \eqref{wj}.
\end{itemize}

\begin{proof}(\emph{of Proposition \ref{propba}})
The solution $w_{j}(r)$ given in \eqref{solj} writes
\begin{equation}
w_{j}(r)=\frac{\pi}{8i}r^{-1/2}\Big(\int_{0}^{r}\chi_{\lambda}(s)f_{j}(s)s^{3/2}(\bar{H}_{\nu_{j}}(\lambda s)-\frac{\bar{H}_{\nu_{j}}(\lambda)}{H_{\nu_{j}}(\lambda)}H_{\nu_{j}}(\lambda s))dsH_{\nu_{j}}(\lambda r)+
\end{equation}
\[
+\int_{r}^{\infty}\chi_{\lambda}(s)f_{j}(s)s^{3/2}H_{\nu_{j}}(\lambda s)ds (\bar{H}_{\nu_{j}}(\lambda r)-\frac{\bar{H}_{\nu_{j}}(\lambda)}{H_{\nu_{j}}(\lambda)}H_{\nu_{j}}(\lambda r))\Big)
\]
thus in order to obtain estimates for $\|w_{j}(r)\|_{L^{2}([1,1+\lambda^{-\alpha}])}$ we have to determine bounds for $\|H_{\nu_{j}}(\lambda r)\|_{L^{2}([1,1+\lambda^{-\alpha}])}$ since we have for some constant $C>0$
\begin{equation}\label{borne}
\|w_{j}(r)\|_{L^{2}([1,1+\lambda^{-\alpha}])}\leq\|f_{j}(r)\|_{L^{2}([1,1+\lambda^{-\alpha}])}\|H_{\nu_{j}}(\lambda r)\|^{2}_{L^{2}([1,1+\lambda^{-\alpha}])}.
\end{equation}
We consider separately two regimes:
\begin{enumerate}
\item For $\frac{\nu_{j}}{\lambda}\ll 1$ we use \eqref{Ai} to obtain immediately
\begin{equation}\label{bornepetit}
\|H_{\nu_{j}}(\lambda r)\|^{2}_{L^{2}([1,1+\lambda^{-\alpha}])}\lesssim\lambda^{-1-\alpha}.
\end{equation}

\item For $0<\epsilon_{0}\leq c_{j}:=\frac{\nu_{j}}{\lambda}\leq1-\epsilon_{0}$ for some fixed $\epsilon_{0}>0$ we use the expansions \eqref{Aiia}, \eqref{Aiib}: set
\[
\cos\tau=\frac{c_{j}}{r},\quad  dr=c_{j}\frac{\sin\tau}{\cos^{2}\tau}d\tau, \quad \tau\in[\arccos c_{j},\arccos\frac{c_{j}}{1+\lambda^{-\alpha}}]=:[\tau_{j,0},\tau_{j,1}].
\]
Then we have
\begin{equation}\label{bornemoyen}
\|H_{\nu_{j}}(\lambda r)\|^{2}_{L^{2}([1,1+\lambda^{-\alpha}])}=\frac{2}{\pi\lambda}\int_{\tau_{j,0}}^{\tau_{j,1}}\frac{1}{\cos\tau}d\tau=\frac{1}{\pi\lambda}\ln\Big(\frac{1+\sin\tau)}{1-\sin\tau}\Big)|_{\tau_{j,0}}^{\tau_{j,1}}
\end{equation}
\[
\simeq\frac{1}{\pi\lambda}\Big((1+\lambda^{-\alpha})^{2}\frac{(1+\sqrt{1-c^{2}_{j}/(1+\lambda^{-\alpha})^{2}})^{2}}{(1+\sqrt{1-c^{2}_{j}})^{2}}-1\Big)
\simeq\frac{1}{\pi\lambda}\frac{\lambda^{-\alpha}}{(1-c^{2}_{j}+\lambda^{-\alpha})^{1/2}+(1-c^{2}_{j})^{1/2}}
\]
and $1-c^{2}_{j}\in [2\epsilon_{0}-\epsilon^{2}_{0},1-\epsilon^{2}_{0}]$ with $0<\epsilon_{0}<1$ fixed. Notice, however, that if one takes $c_{j}\simeq \lambda^{-\beta}$ for some $\beta>0$, then $\tan\theta_{j}\simeq(1-c^{2}_{j})^{1/2}\simeq\lambda^{-\beta/2}$. If $\beta\leq\alpha$ we can estimate \eqref{bornemoyen} by
$\lambda^{-(1-\alpha/2-(\alpha-\beta)}$, otherwise we have the bound $\lambda^{-(1+\alpha/2)}$.
\end{enumerate}
\end{proof}

\begin{proof}(\emph{of Proposition \ref{propsa}})
Here we use Proposition \ref{prophank}. Let $1-\nu_{j}/\lambda=\tau\lambda^{-\beta}$ for $\tau$ in a neighborhood of $1$. For $s\in [1,1+\lambda^{-\alpha}]$ write $z(s)\nu_{j}=s\lambda$, thus $z(s)=\frac{s}{1-\tau\lambda^{-\beta}}$. Since
\begin{equation}\label{gl}
\|w_{j}\| _{L^{2}([1,1+\lambda^{-\alpha}])}\leq C\frac{\lambda^{-\alpha/2}}{|H_{\nu_{j}}(\lambda)|}\|f_{j}\|_{L^{2}([1,1+\lambda^{-\alpha}])}\|H_{\nu_{j}}(\lambda r)\|_{L^{2}([1,1+\lambda^{-\alpha}])}\times
\end{equation}
\[
\times\|J_{\nu_{j}}(\lambda s)Y_{\nu_{j}}(\lambda)-J_{\nu_{j}}(\lambda)Y_{\nu_{j}}(\lambda s)\|_{L^{2}(1,1+\lambda^{-\alpha})},
\]
we shall compute separately each factor in \eqref{gl} (modulo small terms). We have
\begin{equation}\label{A}
\|H_{\nu_{j}}(\lambda r)\|^{2}_{L^{2}(1,1+\lambda^{-\alpha})}=\int_{1}^{1+\lambda^{-\alpha}}|J_{\nu_{j}}(\lambda r)|^{2}+|Y_{\nu_{j}}(\lambda r)|^{2}dr\simeq
\end{equation}
\[
\frac{2}{\pi}\nu^{-1}_{j}\int_{1}^{1+\lambda^{-\alpha}}\frac{1}{(z^{2}(s)-1)^{1/2}ds}\simeq\frac{2}{\pi}\nu^{-1}_{j}\frac{\lambda^{-\alpha}}{\lambda^{-\alpha}+2\tau\lambda^{-\beta}}\simeq
\left\{\begin{array}{ll}
\lambda^{-1-(\alpha-\beta)}, \quad \text{si} \quad 0\leq\beta\leq \alpha,\\
\lambda^{-1},\quad \text{si}, \quad \beta>\alpha,
\end{array}
\right.
\]
\begin{equation}\label{B}
|H_{\nu_{j}}(\lambda)|^{2}=|J_{\nu_{j}}(\lambda)|^{2}+|Y_{\nu_{j}}(\lambda)|^{2}\simeq \frac{2}{\pi}\nu_{j}^{-1}\frac{1}{(z^{2}(1)-1)^{1/2}}\simeq \lambda^{-1+\beta/2},
\end{equation}
while for the factor in the second line in \eqref{gl} we have 
\begin{equation}\label{C}
\|J_{\nu_{j}}(\lambda s)Y_{\nu_{j}}(\lambda)-J_{\nu_{j}}(\lambda)Y_{\nu_{j}}(\lambda s)\|_{L^{2}(1,1+\lambda^{-\alpha})}\leq
\end{equation}
\[
\frac{2}{\pi}\nu^{-1}_{j}\frac{1}{(z^{2}(1)-1)^{1/4}}(\int_{1}^{1+\lambda^{-\alpha}}\frac{1}{(z^{2}(s)-1)^{1/2}}ds)^{1/2}\simeq
\left\{\begin{array}{ll}
\lambda^{-1-\alpha/2+\beta/4}, \quad \text{si} \quad  0\leq\beta\leq \alpha,\\
\lambda^{-1+\beta/4},\quad \text{si}, \quad \beta>\alpha.
\end{array}
\right.
\]
From \eqref{gl}, \eqref{A}, \eqref{B}, \eqref{C} we deduce
\begin{equation}
\|w_{j}\| _{L^{2}([1,1+\lambda^{-\alpha}])}\leq C\|f_{j}\|_{L^{2}([1,1+\lambda^{-\alpha}])}\times
\left\{\begin{array}{ll}
\lambda^{-1-\alpha/2-(\alpha-\beta)}, \quad \text{si} \quad  0\leq\beta\leq \alpha,\\
\lambda^{-1-\alpha/2},\quad \text{si}, \quad \beta>\alpha.
\end{array}
\right.
\end{equation}
\end{proof}

\textbf{Neumann:} \emph{As far as the Neumann problem is concerned, we must solve the
problem
\begin{equation}\label{2nuN}
\Big(\partial^{2}_{r}+ \frac{2}{r}\partial_{r}+
(\lambda^{2}-\frac{\nu^{2}}{r^{2}})\Big)w_{j}(r)=\chi(\lambda^{\alpha}(r-1))f_{j}(r),\quad
\frac{\partial w_{j}}{\partial r}(1)=0,
\end{equation}
which gives, after performing similar computations
\begin{equation}\label{soln}
w^{N}_{j}(r)=\frac{\pi}{8i}r^{-1/2}H_{\nu_{j}}(\lambda
r)\Big(\frac{\bar{H}^{'2}_{\nu_{j}}(\lambda)}{|H'_{\nu_{j}}(\lambda)|^{2}}\int_{1}^{\infty}\chi_{\lambda}(s)f_{j}(s)s^{3/2}H_{\nu_{j}}(\lambda
s)ds-
\end{equation}
\[
-\int_{1}^{r}f_{j}(s)s^{3/2}\bar{H}_{\nu_{j}}(\lambda
s)ds\Big)
 -\frac{\pi}{8i}r^{-1/2}\bar{H}_{\nu_{j}}(\lambda
r)\int_{r}^{\infty}\chi_{\lambda}(s)f_{j}(s)s^{3/2}H_{\nu_{j}}(\lambda.
s)ds,
\]
Propositions \ref{propba}, \ref{propsa} hold true for the Neumann case too.}

\section{Proof of Theorem $1$}\label{secthm1}
\subsection{Reduction to the Helmholtz equation}
\begin{prop}
Consider the Helmholtz equation
\begin{equation}\label{helmholtz}
\left\{
             \begin{array}{ll}
                (\Delta_{D}+\lambda^{2})w=\chi_{\lambda}f\quad on \quad \Omega,\\
               w|_{\partial\Omega}=0,\\
                \end{array}
                \right.
\end{equation}
with the Sommerfield "radiation condition" (where $r=|x|$)
\begin{equation}\label{radiation}
r(\partial_{r}w-i\lambda w)\rightarrow_{r\rightarrow\infty}0.
\end{equation}
In order to prove Theorem \ref{thm1} it is enough to establish estimates for the $L^{2}$ norms of the Helmholtz equation \eqref{helmholtz}.
\end{prop}
\begin{proof}
Consider the inhomogeneous Schr\"{o}dinger equation
\begin{equation}\label{four}
(i\partial_{t}+\Delta_{D})v(t,x)= \chi_{\lambda}g(t,x), \quad
v|_{t=0}=0.
\end{equation}
We denote by $v^{\pm}$ the solutions to the equations
\begin{equation}\label{fourpm}
(i\partial_{t}+\Delta_{D})v^{\pm}(t,x)=
\chi_{\lambda}g(t,x)\textbf{1}_{\{\pm t>0\}}, \quad
v^{\pm}|_{t=0}=0.
\end{equation}
For $\epsilon>0$ and $\pm t>0$ we define
$v^{\pm}_{\epsilon}=e^{\mp\epsilon t}v^{\pm}$ which satisfy
\begin{equation}\label{foureps}
\left\{
             \begin{array}{ll}
                (i\partial_{t}+\Delta_{D}\pm i\epsilon)v_{\epsilon}^{\pm}=1_{\{\pm t>0\}}e^{-\epsilon t}\chi_{\lambda}g(t,x),\\
               v_{\epsilon}^{\pm}|_{t=0}=0,\quad
               v_{\epsilon}^{\pm}|_{\partial\Omega}=0.\\
                \end{array}
                    \right.
\end{equation}
After performing the Fourier transform $\mathcal{F}$ with respect to the time variable $t$, the equation
\eqref{foureps} becomes
\begin{equation}\label{eps}
\left\{
             \begin{array}{ll}
(\Delta_{D}-\tau\pm i\epsilon)\hat{v}^{\pm}_{\epsilon}(\tau,x)=
\chi_{\lambda}\mathcal{F}(1_{\{\pm t>0\}}e^{-\epsilon t}\chi_{\lambda}{g})(\tau,x),\\
\hat{v}^{\pm}_{\epsilon}|_{\partial\Omega}=0,
            \end{array}
                 \right.
\end{equation}
where $\tau$ denotes the dual variable of $t$ and we deduce
\begin{equation}\label{outgo}
\chi_{\lambda}\hat{v}_{+}^{\epsilon}(\tau,x)=\chi_{\lambda}(\Delta_{D}-\tau+i\epsilon)^{-1}
\chi_{\lambda}\mathcal{F}(1_{\{\pm t>0\}}e^{-\epsilon t}\chi_{\lambda}{g})(\tau,x).
\end{equation}
Since $-\Delta_{D}$ is a positive self-adjoint operator, the resolvent $(-\Delta_{D}-z)^{-1}$ is analytic in $\mathbb{C}\setminus\mathbb{R}_{+}$. Since the spectrum of $-\Delta_{D}$ is confined to the positive real axis it is convenient to introduce $\lambda^{2}\in\mathbb{R}_{+}$ as a spectral parameter. Notice, however, that there are two manners to approach $\lambda^{2}>0$ in $\mathbb{C}\setminus\mathbb{R}_{+}$, choosing the positive imaginary part, which corresponds  to considering $\lambda^{2}+i\epsilon$, or the negative imaginary part, which corresponds to $\lambda^{2}-i\epsilon$. The "physical" choice corresponds to the limiting absorption principal and consists of taking $\lambda^{2}+i\epsilon$. In some sense, the limiting absorption principal allows to recover the "sense of time". If one replaces $\lambda^{2}$ by $\lambda^{2}\pm i\epsilon$, $\epsilon>0$, then $(\lambda^{2}\pm i\epsilon)$ belongs to the resolvent of the Laplace operator on $\Omega$ with Dirichlet boundary conditions on $\partial\Omega$ and we can let $\epsilon$ tend to $0$ in \eqref{outgo} (since the
operator
$\chi_{\lambda}(\Delta_{D}-\tau+i\epsilon)^{-1}\chi_{\lambda}$ has a
limit as $\epsilon\rightarrow 0$) and so we can express the
Fourier transform of the unique solutions $v^{\pm}$ of \eqref{fourpm} as
\begin{equation}\label{outg}
\chi_{\lambda}\hat{v}^{\pm}(\tau,x)=\lim_{\epsilon\rightarrow 0}\chi_{\lambda}(\Delta_{D}-\tau\pm i\epsilon)^{-1}
\chi_{\lambda}\mathcal{F}(1_{\{\pm t>0\}}e^{-\epsilon t}\chi_{\lambda}{g})(\tau,x).
\end{equation}
\begin{rmq}
Notice that here we used the fact that for $-\frac{\tau}{\lambda^{2}}$ away from a
neighborhood of $1$ we have the bounds
\[
\|\chi_{\lambda}\hat{v}^{+}\|_{L^{2}}\leq \frac{C\|\chi_{\lambda}\hat{g}\|_{L^{2}}}{|\tau|+\lambda^{2}}. 
\]
\end{rmq}
We conclude using that if $w_{\epsilon}=\hat{v}^{+}_{\epsilon}$, $f_{\epsilon}=\mathcal{F}(1_{\{t>0\}}e^{-\epsilon t}\chi_{\lambda}{g})$ the following holds
\begin{prop}(\cite{tay96})
As $\epsilon\rightarrow 0$, $w_{\epsilon}$ converges to the unique solution of \eqref{helmholtz}-\eqref{radiation}, where $f=\lim_{\epsilon\rightarrow 0}f_{\epsilon}$.
\end{prop}
It remains to notice that if $\mathcal{H}$ is any Hilbert space,
than the Fourier transform defines an isometry on
$L^{2}(\mathbb{R},\mathcal{H})$.
\end{proof}

\subsection{Smoothing effect}
In this section we prove Theorem \ref{thm1}:
\begin{itemize}
\item We prove \eqref{i} for $s=0$. Let $\lambda>0$, $w$ and $g$
be such that \eqref{2d} holds. We multiply \eqref{2d} by
$\chi_{\lambda}\bar{w}$ and we integrate on $\Omega$
\begin{equation}
\int\chi_{\lambda}|\nabla w|^{2}dx=
\lambda^{2}\int\chi_{\lambda}|w|^{2}dx-2\lambda^{\alpha}\int<\nabla
w,\nabla\chi(\lambda^{\alpha}(|x|-1))>\bar{w}dx-\int
\chi_{\lambda}^{2}f\bar{w}dx,
\end{equation}
thus, using that $0\leq\chi_{\lambda}\leq 1$, $\chi_{\lambda}\leq\chi_{\lambda}^{2}$, $\nabla\chi(\lambda^{\alpha}(|x|-1))\lesssim\chi_{\lambda}$, we have for all $\delta>0$
\begin{equation}
\int|\chi_{\lambda}\nabla
w|^{2}dx\lesssim\lambda^{2}\int|\chi_{\lambda}w|^{2}dx
+\delta\int|\chi_{\lambda}f|^{2}dx 
+\frac{1}{4\delta}\int|\chi_{\lambda}w|^{2}dx+\lambda^{\alpha}\int\bar{w}|\nabla w|\chi_{\lambda}^{2}dx.
\end{equation}
Since for all $\delta_{1}>0$ one has
\[
\int\bar{w}|\nabla w|\chi_{\lambda}^{2}dx\leq(\int|\chi_{\lambda}\nabla w|^{2}dx)^{1/2}(\int|\chi_{\lambda}w|^{2}dx)^{1/2}\leq \delta_{1}\|\chi_{\lambda}\nabla w\|^{2}_{L^{2}(\Omega)}+\frac{1}{4\delta_{1}}\|\chi_{\lambda}w\|^{2}_{L^{2}(\Omega)},
\]
taking $\delta=\lambda^{-\alpha}$, $\delta_{1}=\lambda^{-\alpha}/2$ together with $\int|\chi_{\lambda}w|^{2}\lesssim\lambda^{-2-\alpha}\int|\chi_{\lambda}f|^{2}$ (according to the computations made in the preceding section) we deduce
\begin{equation}
\int|\chi_{\lambda}\nabla
w|^{2}dx\lesssim\lambda^{-\alpha}\int|\chi_{\lambda}f|^{2}dx.
\end{equation}
Thus we have obtained
\begin{equation}\label{ze}
\|\chi_{\lambda}w\|_{H^{1}_{D}(\Omega)}
\lesssim\lambda^{-\frac{\alpha}{2}}\|\chi_{\lambda}f\|_{L^{2}(\Omega)}.
\end{equation}
Dualizing \eqref{ze} we find
\begin{equation}\label{un}
\|\chi_{\lambda}w\|_{L^{2}(\Omega)}\lesssim\lambda^{-\frac{\alpha}{2}}
\|\chi_{\lambda}f\|_{H_{D}^{-1}(\Omega)},
\end{equation}
which will yield \eqref{i} for $s=-1$. Now we prove \eqref{i} for
$s=1$. Let again $w$ and $f$ be such that \eqref{2d} holds and
let $\tilde{\chi}_{\lambda}$ be a smooth cutoff function equal to $1$
on the support of $\chi_{\lambda}$. Write
\begin{equation}
\|\chi_{\lambda}w\|_{H_{D}^{2}(\Omega)}\thickapprox\|\chi_{\lambda}w\|_{H_{D}^{1}(\Omega)}+
\|\Delta_{D}(\chi_{\lambda}w)\|_{L^{2}(\Omega)}.
\end{equation}
Since $\|\chi_{\lambda}w\|_{H_{D}^{1}(\Omega)}$ can be estimated
by means of \eqref{ze}, we only need to obtain bounds for
$\|\Delta_{D}(\chi_{\lambda}w)\|_{L^{2}(\Omega)}$. We write
\begin{equation}\label{norma}
\Delta_{D}(\chi_{\lambda}w)=\chi_{\lambda}\Delta_{D}w+[\Delta_{D},\chi_{\lambda}]\tilde{\chi}_{\lambda}w.
\end{equation}
The commutator $[\Delta_{D},\chi_{\lambda}]$ is bounded from
$H_{D}^{1}(\Omega)$ to $L^{2}(\Omega)$ with norm less then
$C\lambda^{\alpha}$ and we have chosen $\alpha<1$ ($C$ does not
depend on $\lambda$), while the first term in the right hand side
of \eqref{norma} is in $H_{D}^{1}(\Omega)$ since
$\Delta_{D}w=\chi_{\lambda}f-\lambda^{2}w$ and satisfies
\eqref{2d} with $\chi_{\lambda}f$ replaced by
$\Delta_{D}(\chi_{\lambda}f)$. Therefore, we can apply \eqref{ze}
in order to deduce the inequality \eqref{i} for the Fourier
transforms in time of $u$ and $f$ which appear in the proposition.
Since we have obtained the result for $s=-1$ and $s=1$ we can use
an interpolation argument to get it for $s\in [-1,1]$.

\item We turn to the proof of \eqref{ii} for $s=0$. If we denote by
$A_{\lambda}$ the operator which to a given $u_{0}\in
L^{2}(\Omega)$ associates
$\chi_{\lambda}e^{it\Delta_{D}}\psi(\frac{-\Delta_{D}}{\lambda^{2}})u_{0}$,
we need to prove that $A_{\lambda}$ is bounded from
$L^{2}(\Omega)$ to $L^{2}_{T}H_{D}^{\frac{1}{2}}(\Omega)$ with the
norm less than $C\lambda^{-\frac{\alpha}{4}}$ for some constant $C$ independent of $\lambda$, which in turn is
equivalent to the continuity of the adjoint operator,
\begin{equation}
A^{*}_{\lambda}(f)=\int_{0}^{T}\psi(\frac{-\Delta_{D}}{\lambda^{2}})e^{-i\tau\Delta_{D}}
\chi_{\lambda}f(\tau)d\tau,
\end{equation}
from $L^{2}_{T}H_{D}^{-\frac{1}{2}}(\Omega)$ to $L^{2}(\Omega)$
with the norm bounded by $C\lambda^{-\frac{\alpha}{4}}$, which is
equivalent to showing that the operator
$A_{\lambda}A_{\lambda}^{*}$ defined by
\begin{equation}
(A_{\lambda}A^{*}_{\lambda}f)(t)=\int_{0}^{T}\chi_{\lambda}e^{it\Delta_{D}}\psi^{2}
(\frac{-\Delta_{D}}{\lambda^{2}})e^{-i\tau\Delta_{D}}\chi_{\lambda}f(\tau)d\tau
\end{equation}
is continuous from $L^{2}_{T}H_{D}^{-\frac{1}{2}}(\Omega)$ to
$L^{2}_{T}H_{D}^{\frac{1}{2}}(\Omega)$ and its norm is bounded
from above by $C\lambda^{-\frac{\alpha}{2}}$. We write
$(A_{\lambda}A^{*}_{\lambda}f)(t)$ as a sum
\begin{equation}\label{aa}
(A_{\lambda}A^{*}_{\lambda}f)(t)=\int_{0}^{t}\chi_{\lambda}\psi(\frac{-\Delta_{D}} {\lambda^{2}})
e^{i(t-\tau)\Delta_{D}}\psi(\frac{-\Delta_{D}}{\lambda^{2}})\chi_{\lambda}f(\tau)d\tau
+
\end{equation}
\[
+\int_{t}^{T}\chi_{\lambda}\psi(\frac{-\Delta_{D}}{\lambda^{2}})e^{i(t-\tau)\Delta_{D}}
\psi(\frac{-\Delta_{D}}{\lambda^{2}})\chi_{\lambda}f(\tau)d\tau.
\]
Hence, in order to conclude it is sufficient to apply \eqref{i}
with $s=-\frac{1}{2}$ together with time inversion, since the second
term on the right hand side of \eqref{aa} will solve the same
problem with initial data $u|_{t=T}=u_{0}$. 

We prove now \eqref{ii} for $s=1$. 
\begin{lemma}\label{lemlapla}
We have
\begin{equation}
\|\Delta_{D}\Big(\chi_{\lambda}
e^{it\Delta_{D}}\psi(\frac{-\Delta_{D}}{\lambda^{2}})u_{0}\Big)\|_{L^{2}_{T}H^{-\frac{1}{2}}_{D}
(\Omega)}\lesssim
\lambda^{1-\frac{\alpha}{4}}\|\chi_{\lambda}\psi(\frac{-\Delta_{D}}{\lambda^{2}})
u_{0}\|_{L^{2}(\Omega)}.
\end{equation}
\end{lemma}
\begin{cor}\label{corlapla}
Lemma \ref{lemlapla} yields
\begin{equation}
\|\chi_{\lambda}
e^{it\Delta_{D}}\psi(\frac{-\Delta_{D}}{\lambda^{2}})u_{0}\|_{L^{2}_{T}H^{\frac{3}{2}}_{D}
(\Omega)}\lesssim
\lambda^{1-\frac{\alpha}{4}}\|\chi_{\lambda}\psi(\frac{-\Delta_{D}}{\lambda^{2}})
u_{0}\|_{L^{2}(\Omega)}.
\end{equation}
\end{cor}
Corollary \ref{corlapla} and an interpolation argument now yield for $\theta\in[0,1]$
\begin{equation}
\|\chi_{\lambda}
e^{it\Delta_{D}}\psi(\frac{-\Delta_{D}}{\lambda^{2}})u_{0}\|_{L^{2}_{T}H^{\frac{3\theta}{2}}_{D}(\Omega)}\lesssim
\lambda^{\frac{3\theta}{2}-\frac{1}{2}-\frac{\alpha}{4}}\|\chi_{\lambda}
\psi(\frac{-\Delta_{D}}{\lambda^{2}})u_{0}\|_{L^{2}(\Omega)},
\end{equation}
achieving the proof of Theorem \ref{thm1}.
\begin{proof}(\emph{of Lemma \ref{lemlapla}})
Write
\begin{equation}\label{dlt}
(-\Delta_{D}+1)\Big(\chi_{\lambda}
e^{it\Delta_{D}}\psi(\frac{-\Delta_{D}}{\lambda^{2}})u_{0}\Big)=
\chi_{\lambda}
(-\Delta_{D}+1)\Big(e^{it\Delta_{D}}\psi(\frac{-\Delta_{D}}{\lambda^{2}})u_{0}\Big)-
\end{equation}
\[
-[\Delta_{D},\chi_{\lambda}]
\Big(e^{it\Delta_{D}}\psi(\frac{-\Delta_{D}}{\lambda^{2}})u_{0}\Big).
\]
\begin{itemize}
\item For the first term in the right hand side of \eqref{dlt} we show that the operator
\begin{equation}
B_{\lambda}:=\Big(H_{D}^{1}(\Omega)\ni
u_{0}\rightarrow\chi_{\lambda}(-\Delta_{D}+1)\Big(e^{it\Delta_{D}}\psi
(\frac{-\Delta_{D}}{\lambda^{2}})u_{0}\Big)\in
L^{2}_{T}H^{-\frac{1}{2}}_{D}(\Omega)\Big)
\end{equation}
is continuous and its norm from $H^{1}_{D}(\Omega)$ to $L^{2}_{T}H^{-\frac{1}{2}}_{D}(\Omega)$ is bounded from above by
$C\lambda^{-\frac{\alpha}{4}}$ for some constant $C$ independent of $\lambda$, or equivalently that
\begin{equation}\label{bb}
\begin{split}
\Big(B_{\lambda}(-\Delta_{D}+1)^{-1}B^{*}_{\lambda}f\Big)(t)=\int_{0}^{T}\chi_{\lambda}(-\Delta_{D}+1)\psi^{2}(\frac{-\Delta_{D}}{\lambda^{2}})
e^{i(t-\tau)\Delta_{D}}\chi_{\lambda}f(\tau)d\tau=\\
=(-\Delta_{D}+1)(A_{\lambda}A_{\lambda}^{*}f)(t)+[\Delta_{D},\chi_{\lambda}]\int_{0}^{T}
e^{i(t-\tau)\Delta_{D}}\psi^{2}(\frac{-\Delta_{D}}{\lambda^{2}})\chi_{\lambda}
f(\tau)d\tau
\end{split}
\end{equation}
is bounded from $L^{2}_{T}H^{\frac{1}{2}}_{D}(\Omega)$ to
$L^{2}_{T}H^{-\frac{1}{2}}_{D}(\Omega)$ by
$C\lambda^{-\frac{\alpha}{2}}$. Here $A_{\lambda}$ is the operator
introduced in the proof of the case $s=0$. For the first term in the right hand side of \eqref{bb} we
apply \eqref{i} with $s=\frac{1}{2}$ and we obtain a bound $C\lambda^{-\frac{\alpha}{2}}$, while for the second term we make use of \eqref{ii} with $s=-\frac{1}{2}$ and of the fact that
$[\Delta_{D},\chi_{\lambda}]$ is bounded from
$H_{D}^{\frac{1}{2}}(\Omega)$ to $H_{D}^{-\frac{1}{2}}(\Omega)$
with a norm bounded by $C\lambda^{\alpha}$ in order to find a bound  from $L^{2}_{T}H^{\frac{1}{2}}_{D}(\Omega)$
to $L^{2}_{T}H^{-\frac{1}{2}}_{D}(\Omega)$ of at most
$C\lambda^{\frac{\alpha}{2}-1}\leq C\lambda^{ -\frac{\alpha}{2}}$.

\item The second term in the right hand side of \eqref{dlt} gives
\begin{equation}\label{secterm}
\|[\Delta_{D},\chi_{\lambda}]e^{it\Delta_{D}}\psi(\frac{-\Delta_{D}}{\lambda^{2}})u_{0}\|_{L^{2}_{T}H^{-\frac{1}{2}}_{D}(\Omega)}\lesssim
\end{equation}
\[
\lesssim\lambda^{2\alpha}\|\chi''_{\lambda}e^{it\Delta_{D}}\psi(\frac{-\Delta_{D}}{\lambda^{2}})u_{0}\|
_{L^{2}_{T}H^{-\frac{1}{2}}_{D}(\Omega)}+2\lambda^{\alpha}\|\chi_{\lambda}'\nabla\Big(e^{it\Delta_{D}}
\psi(\frac{-\Delta_{D}}{\lambda^{2}})u_{0}\Big)\|_{L^{2}_{T}H^{-\frac{1}{2}}_{D}(\Omega)}.
\]
For evaluating the first term in the last sum we use again \eqref{ii} with $s=0$
\begin{equation}
\lambda^{2\alpha}\|\chi''_{\lambda}e^{it\Delta_{D}}\psi(\frac{-\Delta_{D}}{\lambda^{2}})u_{0}\|_{L^{2}_{T}
H^{\frac{1}{2}}_{D}(\Omega)}\lesssim
\lambda^{2\alpha-\frac{\alpha}{4}}\|\tilde{\chi}_{\lambda}
\psi(\frac{-\Delta_{D}}{\lambda^{2}})u_{0}\|_{L^{2}(\Omega)}.
\end{equation}
where $\breve{\chi}_{\lambda}$ is a smooth cutoff function such that $\tilde{\chi}_{\lambda}$ is equal to $1$
on the support of $\chi_{\lambda}$ and we conclude since $\alpha<1$. For the second term we have
\begin{equation}
\|\chi_{\lambda}'\nabla\Big(e^{it\Delta_{D}}
\psi(\frac{-\Delta_{D}}{\lambda^{2}})u_{0}\Big)\|^{2}_{L^{2}(\Omega)}=
-\int\chi_{\lambda}'^{2}(\Delta_{D} u)\bar{u}
-2\lambda^{\alpha}Re\int\chi_{\lambda}'\chi''_{\lambda}(\nabla
u)\bar{u}
\end{equation}
\[
\leq\lambda^{2}\int\chi^{'2}_{\lambda}|\Big(e^{it\Delta_{D}}\tilde{\psi}(\frac{-\Delta_{D}}{\lambda^{2}})u_{0}\Big)\bar{u}|+
2\lambda^{\alpha+1/2}\|\tilde{\chi}_{\lambda}\|^{2}_{L^{2}(\Omega)}\lesssim \lambda^{2}\|\tilde{\chi}_{\lambda}u\|^{2}_{L^{2}(\Omega)},
\]
where we have set $\tilde{\psi}(x)=x\psi(x)$.
From \eqref{ii} with $s=0$ that we have already established,
we find, since $\alpha<1$,
\begin{equation}
\lambda^{\alpha}\|\chi_{\lambda}'\nabla\Big(e^{it\Delta_{D}}\psi(\frac{-\Delta_{D}}{\lambda^{2}})u_{0}\Big)\|_{L^{2}_{T}H^{-\frac{1}{2}}_{D}(\Omega)}
\lesssim\lambda^{\alpha}\|\tilde{\chi}_{\lambda}u\|_{L^{2}_{T}H^{\frac{1}{2}}(\Omega)}
\end{equation}
\[
\lesssim\lambda^{\frac{3\alpha}{4}}\|\tilde{\chi}_{\lambda}\psi(\frac{-\Delta_{D}}{\lambda^{2}})u_{0}
\|_{L^{2}(\Omega)}\lesssim\lambda^{1-\frac{\alpha}{4}}\|\tilde{\chi}_{\lambda}\psi(\frac{-\Delta_{D}}{\lambda^{2}})u_{0}\|
_{L^{2}(\Omega)}.
\]
\end{itemize}
\end{proof}
\end{itemize}

\textbf{Neumann:} \emph{All of the above results remain valid if we consider the Neumann
Laplacian $\Delta_{N}$: in fact let $w$ and $f$ be such that
\eqref{2d} holds. Using the same strategy as before we get:
\begin{equation}
\int_{\partial\Omega}\partial_{n}
u(\chi_{\lambda}\bar{u})d\sigma-\int_{\Omega}\chi_{\lambda}|\nabla
w|^{2}dx+\lambda^{2}\int_{\Omega}\chi_{\lambda}|w|^{2}dx-\lambda^{\alpha}\int_{\Omega}<\nabla
w,\nabla\chi_{\lambda}>\tilde{\chi}_{\lambda}\bar{w}dx=\int_{\Omega}\chi_{\lambda}^{2}f\bar{w}dx.
\end{equation}
Notice that, due to the Neumann boundary
conditions, the first term in the left hand side vanishes, so the
computations are almost the same as in the previous case. However
we must pay a little attention to the spectrum of the Neumann
Laplacian: for, we have to introduce a spectral cut-off $\phi\in C^{\infty}_{0}$ equal to $1$ close to $0$ and decompose
\begin{equation}
w=\phi(-\Delta_{N})w+(1-\phi(-\Delta_{N}))w,\quad
f=\phi(-\Delta_{N})f+(1-\phi(-\Delta_{N}))f,
\end{equation}
\[
u_{0}=\phi(-\Delta_{N})u_{0}+(1-\phi(-\Delta_{N}))u_{0}.
\]
We can then rewrite the above proof with $w$, $f$, $u_{0}$
replaced by $(1-\phi(-\Delta_{N}))w$, $(1-\phi(-\Delta_{N}))f$,
$(1-\phi(-\Delta_{N}))u_{0}$ in order to obtain similar estimates
in Theorem \ref{thm1}. To deal with the contributions of the
remaining terms we use the fact that in this situation the $L^{2}$
and $H_{N}^{k}$ norms are equivalent.}

\section{Estimates away from the obstacle}\label{away}
In this section we obtain bounds away from the obstacle. The main
idea is to construct a new function $v=\phi u$ which
will solve a problem with a nonlinearity supported in a compact
set; under these assumptions, it is proved that the free evolution
satisfies the usual Strichartz bounds (see \cite{stta02}). However, we have to take
into account the fact that the neighborhood outside of which we
will apply this result is of size $\lambda^{-\alpha}$ and thus we
will "lose" $\alpha$ derivative; however, this will not pose any
problem if $\alpha$ is chosen small enough, since it will be
covered by the loss of derivatives near the boundary. After a change of variables we can assume that $\Omega=\{x_{3}>0\}$ and thus $x_{3}$ defines the distance to the boundary. We define
$\phi\in C^{\infty}(\bar{\Omega})$ by
$\phi(x)|_{x_{3}\leq 1}=x_{3}$,
$\phi(x)|_{x_{3}\geq 2}=1$.
\begin{prop}
We have
\begin{equation}\label{aw}
\|(1-\chi_{\lambda})e^{it\Delta_{D}}\psi(\frac{-\Delta_{D}}{\lambda^{2}})u_{0}\|_{L^{2}_{T}L^{\frac{2d}{d-2}}(\Omega)}\lesssim
\lambda^{\alpha}\|\psi(\frac{-\Delta_{D}}{\lambda^{2}})u_{0}\|
_{L^{2}(\Omega)}.
\end{equation}
\end{prop}
We set $v=\phi u$, where $u=e^{it\Delta_{D}}\psi(\frac{-\Delta_{D}}{\lambda^{2}})u_{0}$ solves
\eqref{LS}. Then $v$ solves the equation
\begin{equation}\label{3d}
\left\{\begin{array}{ll}
(i\partial_{t}+\Delta_{D})v=(\Delta_{D}\phi)u+2\nabla\phi\nabla
u+\frac{\partial v}{\partial_{n}}|_{\partial\Omega}\otimes\delta_{\partial\Omega}^{(0)}-v|_{\partial\Omega}\otimes
\delta_{\partial\Omega}^{(1)}=[\Delta_{D},\phi]u,\\
v|_{t=0}=\phi u|_{t=0},\quad
v|_{\partial\Omega}=0,
\end{array}
\right.
\end{equation}
where $\delta$ is the Dirac measure on $\partial\Omega$. It can
easily be seen that the last two terms vanish. We have
$[\Delta_{D},\phi]u\in
L^{2}_{T}H^{\frac{-1}{2}}_{comp}(\Omega)$ and (see \cite[Prop.2.7]{bgt03})
\begin{equation}\label{st}
\|[\Delta_{D},\phi]e^{it\Delta_{D}}\psi(\frac{-\Delta_{D}}{\lambda^{2}})u_{0}\|_{L^{2}_{T}H^{\frac{-1}{2}}_{D}(\Omega)}\lesssim
\|e^{it\Delta_{D}}\psi(\frac{-\Delta_{D}}{\lambda^{2}})u_{0}\|_{L^{2}_{T}H^{\frac{1}{2}}_{D}(\Omega)}
\lesssim\|\psi(\frac{-\Delta_{D}}{\lambda^{2}})u_{0}\|_{L^{2}(\Omega)}.
\end{equation}
The inhomogeneous part of the
equation \eqref{3d} satisfied by $v$ has compact spatial support
and therefore we can employ \cite[Thm.3]{stta02} (for $p=2$) and
\cite[Prop.2.10]{bgt03} in order to obtain Strichartz estimates
without losses for $v$
\begin{equation}\label{nlos}
\|\phi e^{it\Delta_{D}}\psi(\frac{-\Delta_{D}}{\lambda^{2}})u_{0}\|_{L^{p}_{T}L^{q}(\Omega)}\lesssim
\|\psi(\frac{-\Delta_{D}}{\lambda^{2}})u_{0}\|_{L^{2}(\Omega)},
\end{equation}
for every $(p,q)$ $d$-admissible pair, which in turn yields 
\begin{equation}\label{nnl}
\lambda^{-\alpha}\|(1-\chi_{\lambda})e^{it\Delta_{D}}\psi(\frac{-\Delta_{D}}{\lambda^{2}})u_{0}\|_{L^{p}_{T}L^{q}(\Omega)}\lesssim
\|(1-\chi_{\lambda})\phi e^{it\Delta_{D}}\psi(\frac{-\Delta_{D}}{\lambda^{2}})u_{0}\|_{L^{p}_{T}L^{q}(\Omega)}
\end{equation}
\[
\lesssim\|\phi e^{it\Delta_{D}}\psi(\frac{-\Delta_{D}}{\lambda^{2}})u_{0}\|_{L^{p}_{T}L^{q}(\Omega)}\lesssim
\|\psi(\frac{-\Delta_{D}}{\lambda^{2}})u_{0}\|_{L^{2}(\Omega)}.
\]
In particular, for $p=2$ we get \eqref{aw} from \eqref{nnl}.

\textbf{Neumann:} \emph{When dealing with the Neumann problem, this approach requires some
adjustments, but the above result remains valid (with $\alpha$
slightly modified). Let $\epsilon>0$ and consider
$\phi|_{x_{3}\leq 1}(x)=x_{3}^{1+\epsilon}$, $\phi|_{x_{3}\geq 2}(x)=1$.
Then $v=\phi u$ solves the equation
\begin{equation}\label{3n}
(i\partial_{t}+\Delta_{D})v=[\Delta_{D},\phi^{1+\epsilon}]u,\quad
v|_{t=0}=\phi u|_{t=0},\quad
\frac{\partial
v}{\partial\nu}|_{\partial\Omega}=(x_{3}^{\epsilon}u)|_{\partial\Omega}+
\phi\frac{\partial u}{\partial\nu}|_{\partial\Omega}=0.
\end{equation}
It's easy to see that one can obtain similar estimates, with
$\alpha$ replaced by $\alpha-\epsilon$; still, this makes no
difference for our purpose, since $\epsilon$ can be chosen as
small as we like.}

\section{Proof of Theorem $2$}\label{secthm2}
In this section we achieve the proof of Theorem \ref{thm2}. 
Taking $s=1/2$ in \eqref{ii} gives
\begin{equation}\label{near}
\|\chi_{\lambda}
e^{it\Delta_{D}}\psi(\frac{-\Delta_{D}}{\lambda^{2}})u_{0}\|_{L^{2}_{T}H^{1}_{D}(\Omega)}\lesssim\lambda^{\frac{1}{2}
-\frac{\alpha}{4}}\|\chi_{\lambda}\psi(\frac{-\Delta_{D}}{\lambda^{2}})
u_{0}\|_{L^{2}(\Omega)}
\end{equation}
from which we deduce
\begin{equation}
\|\chi_{\lambda}e^{it\Delta_{D}}\psi(\frac{-\Delta_{D}}{\lambda^{2}})u_{0}\|_{L^{2}_{T}L^{\frac{2d}{d-2}}(\Omega)}\lesssim
\|\chi_{\lambda}e^{it\Delta_{D}}\psi(\frac{-\Delta_{D}}{\lambda^{2}})u_{0}\|_{L^{2}_{T}H^{1}_{D}(\Omega)}
\lesssim\lambda^{\frac{1}{2}-\frac{\alpha}{4}}\|\psi(\frac{-\Delta_{D}}{\lambda^{2}})
u_{0}\|_{L^{2}(\Omega)}.
\end{equation}
An energy argument yields
\begin{equation}\label{en}
\|\chi_{\lambda}e^{it\Delta_{D}}\psi(\frac{-\Delta_{D}}{\lambda^{2}})u_{0}\|_{L^{\infty}_{T}L^{2}(\Omega)}\lesssim\|
\psi(\frac{-\Delta_{D}}{\lambda^{2}})u_{0}\|_{L^{2}(\Omega)}.
\end{equation}
Interpolation between \eqref{near} and \eqref{en} with weights
$\frac{2}{p}$ and $1-\frac{2}{p}$ respectively yields
\begin{equation}\label{nea}
\|\chi_{\lambda}e^{it\Delta_{D}}\psi(\frac{-\Delta_{D}}{\lambda^{2}})u_{0}\|_{L^{p}_{T}L^{q}(\Omega)}\lesssim\lambda^{\frac{1}{p}(1-\frac{\alpha}{2})}\|
\psi(\frac{-\Delta_{D}}{\lambda^{2}})u_{0}\|_{L^{2}(\Omega)}.
\end{equation}
We have also obtained estimates away from the boundary
\begin{equation}\label{awa}
\|(1-\chi_{\lambda})e^{it\Delta_{D}}\psi(\frac{-\Delta_{D}}{\lambda^{2}})u_{0}\|_{L^{p}_{T}L^{q}(\Omega)}
\lesssim\lambda^{\alpha}\|\psi(\frac{-\Delta_{D}}{\lambda^{2}})u_{0}\|
_{L^{2}(\Omega)}.
\end{equation}
If we take $\alpha=\frac{1}{p}(1-\frac{\alpha}{2})$ we find
$\alpha=\frac{2}{2p+1}$. Now, for $p=2$ this gives
$\alpha=\frac{2}{5}(\leq\frac{2}{3})$ and consequently we have
\begin{equation}\label{p}
\|e^{it\Delta_{D}}\psi(\frac{-\Delta_{D}}{\lambda^{2}})u_{0}\|_{L^{2}_{T}L^{\frac{2d}{d-2}}(\Omega)}
\lesssim\lambda^{\frac{2}{5}}\|\psi(\frac{-\Delta_{D}}{\lambda^{2}})
u_{0}\|_{L^{2}(\Omega)}.
\end{equation}
Interpolation between \eqref{p} and
\begin{equation}\label{ene}
\|e^{it\Delta_{D}}\psi(\frac{-\Delta_{D}}{\lambda^{2}})u_{0}\|_{L^{\infty}_{T}L^{2}(\Omega)}\lesssim\|
\psi(\frac{-\Delta_{D}}{\lambda^{2}})u_{0}\|_{L^{2}(\Omega)},
\end{equation}
with weights $\frac{2}{p}$ and $1-\frac{2}{p}$ yields, for every $(p,q)$ admissible pair, $p\geq 2$
\begin{equation}\label{thm}
\|e^{it\Delta_{D}}\psi(\frac{-\Delta_{D}}{\lambda^{2}})u_{0}\|_{L^{p}_{T}L^{q}(\Omega)}\lesssim\lambda^{\frac{4}{5p}}\|
\psi(\frac{-\Delta_{D}}{\lambda^{2}})u_{0}\|_{L^{2}(\Omega)}.
\end{equation}

\textbf{Neumann:} The case of the Neumann conditions could be handled in the same
way. However in \eqref{awa} we have $\alpha-\epsilon$ 
instead of $\alpha$ so we find 
$\alpha=\frac{2}{5}+\frac{4\epsilon}{5}$ and for every $(p,q)$ $d$-admissible pair and every $\epsilon>0$ we have
\begin{equation}\label{thm}
\|e^{it\Delta_{N}}\psi(\frac{-\Delta_{N}}{\lambda^{2}})u_{0}\|_{L^{p}_{T}L^{q}(\Omega)}\lesssim\lambda^{\frac{4}{5p}+\frac{8\epsilon}{5p}}\|
\psi(\frac{-\Delta_{N}}{\lambda^{2}})u_{0}\|_{L^{2}(\Omega)}.
\end{equation}

\section{Applications}
\subsection{Proof of Theorem \ref{thm3} for non-trapping obstacle}
The proof of Theorem \ref{thm3} relies on the contraction
principle applied to the equivalent integral equation associated
to \eqref{nls} with Dirichlet boundary conditions
\begin{equation}
u(t)=e^{it\Delta_{D}}u_{0}-i\int_{0}^{t}e^{i(t-\tau)\Delta_{D}}F(u(\tau))d\tau.
\end{equation}
The assumptions on $F$ and on the potential $V$ imply the
following pointwise estimates:
\begin{equation}
|F(u)|\lesssim |u|(1+|u|^{2}),\quad |\nabla F(u)|\lesssim |\nabla
u|(1+|u|^{2}),\quad |F(u)-F(v)|\lesssim |u-v|(1+|u|^{2}+|v|^{2}),
\end{equation}
\begin{equation}
|\nabla (F(u)-F(v))|\lesssim |\nabla
(u-v)|(1+|u|^{2}+|v|^{2})+|u-v|(|\nabla u|+|\nabla v|)(1+|u|+|v|).
\end{equation}
The aim is to show that for sufficiently small $T>0$ we can solve
the integral equation by a Picard iteration scheme in the Banach
space $X_{T}:=L_{T}^{\infty}H_{0}^{1}(\Omega)\cap
L_{T}^{p}W^{\sigma,q}(\Omega)$, where $2<p<\frac{12}{5}$ and
$\sigma=1-\frac{4}{5p}$. We equip $X_{T}$ with the norm
\[
\|u\|_{X_{T}}:=\|u\|_{L_{T}^{\infty}H_{0}^{1}(\Omega)}+\|(1-\Delta_{D})^{\frac{\sigma}{2}}u\|
_{L_{T}^{p}L^{q}(\Omega)}.
\]
\begin{rmq}
For $2<p<\frac{12}{5}$ one has the continuous embedding
$W^{1-\frac{4}{5p},q}(\Omega)\subset L^{\infty}(\Omega)$ where
$(p,q)$ is an admissible pair, since in this case
$1-\frac{4}{5p}>\frac{3}{q}$ and Proposition \ref{prop1}
implies $u\in L_{T}^{p}L^{\infty}(\Omega)$ for any $u\in X_{T}$.
\end{rmq}
Define a nonlinear map $\Phi$ as follows
\begin{equation}
\Phi(u)(t):=e^{it\Delta_{D}}u_{0}-i\int_{0}^{t}e^{i(t-\tau)\Delta_{D}}F(u(\tau))d\tau.
\end{equation}
We have to show that $\Phi$ is a contraction in a suitable ball
$B(0,R)$ of $X_{T}$. Using the Minkowski integral inequality
together with an energy argument we have, on the one hand
\[
\|\Phi(u)\|_{L_{T}^{\infty}H_{0}^{1}(\Omega)}\lesssim
\|u_{0}\|_{H_{0}^{1}(\Omega)}+\|F(u)\|_{L_{T}^{1}H_{0}^{1}(\Omega)}\lesssim
\]
\[
(|u_{0}\|_{H_{0}^{1}(\Omega)}+\|u\|_{L_{T}^{1}H_{0}^{1}(\Omega)}\|u\|_{L_{T}^{\infty}H_{0}^{1}(\Omega)}^{2}+
T\|u\|_{L_{T}^{\infty}H_{0}^{1}(\Omega)}+
T^{1-\frac{2}{p}}\|u\|_{L_{T}^{\infty}H_{0}^{1}(\Omega)}\|u\|_{L_{T}^{p}L^{\infty}(\Omega)}^{2})\leq
\]
\begin{equation}\label{est1}
C\Big(\|u_{0}\|_{H_{0}^{1}(\Omega)}+T\|u\|_{X_{T}}+
(T^{1-\frac{2}{p}}+T)\|u\|_{X_{T}}^{3}\Big)
\end{equation}
and also
\[
\|(1-\Delta_{D})^{\frac{\sigma}{2}}\Phi(u)\|_{L_{T}^{p}L^{q}(\Omega)}\lesssim
(\|(1-\Delta_{D})^{\frac{\sigma}{2}}u_{0}\|_{H_{D}^{\frac{4}{5p}}(\Omega)}+
\int_{0}^{T}\|(1-\Delta_{D})^{\frac{\sigma}{2}}(1+|u(s)|^{2})u(s)\|_{H_{D}^{\frac{4}{5p}}(\Omega)})\lesssim
\]
\begin{equation}\label{est2}
(\|u_{0}\|_{H_{0}^{1}(\Omega)}+
\int_{0}^{T}(1+\|u(s)\|_{L^{\infty}}^{2})\|u(s)\|_{H_{0}^{1}(\Omega)})\leq
C\Big(\|u_{0}\|_{H_{0}^{1}(\Omega)}+T\|u\|_{X_{T}}+
T^{1-\frac{2}{p}}\|u\|_{X_{T}}^{3}\Big).
\end{equation}
We chose $R>4C\|u_{0}\|_{H_{0}^{1}(\Omega)}$ and $T$ such that
$4C\Big(T+(T+T^{1-\frac{2}{p}})R^{2}\Big)<1$. In a similar way we
obtain
\begin{equation}\label{lip}
\|\Phi(u)-\Phi(v)\|_{X_{T}}\leq C\|u-v\|_{X_{T}}\Big(T+(T+T^{1-\frac{2}{p}})
(1+\|u\|_{X_{T}}^{2}+\|v\|_{X_{T}}^{2})\Big).
\end{equation}
Taking $T$ eventually smaller such that
$C\Big(T+(T+T^{1-\frac{2}{p}})(1+2R^{2})\Big)<1$, we deduce from
\eqref{est1}, \eqref{est2} and \eqref{lip} that $\Phi$ in a
contraction from $B(0,R)\subset X_{T}$ to $B(0,R)$. Therefore, if
we consider the sequence $\{v_{n}\}_{n\in\mathbb{N}}\subset X_{T}$
such that $v_{0}=u_{0}\in B(0,R)$, $v_{n+1}=\Phi(v_{n})$, then
$v_{n}$ converges in $X_{T}$ to the unique solution in $X_{T}$ of
the integral equation
\begin{equation}
u(t)=e^{it\Delta_{D}}u_{0}-i\int_{0}^{t}e^{i(t-\tau)\Delta_{D}}F(u(\tau))d\tau
\end{equation}
which yields the local well-posedness result. From \eqref{lip} we obtain
the Lipschitz property of the flow map. Using a standard
approximation argument we can derive the conservation laws. Next
due to \eqref{conserv}, the assumption on $V$ ($V(|u|^{2})\geq 0$)
and the Gagliardo-Nirenberg inequality we can extend the local
solution to an arbitrary time interval by reiterating the
local-posedness argument.

\textbf{Neumann:} When we consider the Neumann Laplacian we take $X_{T}$ like before
and we choose $\epsilon$ so small such that the embedding
$W^{1-\frac{4}{5p}-\frac{8\epsilon}{5p},q}(\Omega)\subset
L^{\infty}(\Omega)$ still holds: for $\epsilon<\frac{12-5p}{16}$
we are led to the same conclusion.

\subsection{Proof Theorem \ref{thm2} for a class of trapping obstacles}
In this part we prove Theorem \ref{thm2} for a class of trapping obstacles.
In this case, since there are
trapped trajectories (e.g. any line minimizing the distance
between two obstacles has an unbounded sojourn time), the plain
smoothing effect $H^{\frac{1}{2}}$ does not hold. However, one can
obtain a smoothing effect with a logarithmic loss (see
\cite[Thm1.7,Thm.4.2]{Bu1}).

\textbf{Assumptions:}
We suppose here
$\Theta=\cup_{i=1}^{N}\Theta_{i}$ is the disjoint union of a finite number
of balls $\Theta_{i}=B_{i}(o_{i},r_{i})$ of radius $r_{i}>0$ in $\mathbb{R}^{3}$. We denote by
$k$ the minimum of the curvatures of the spheres
$\mathbb{S}(r_{i})=\partial\Theta_{i}$, that is
$k=\min\{\frac{1}{\sqrt{r_{i}}}, i=\overline{1,N}\}$. Denote by
$L$ the minimum of the distances between two balls. Then, if
$N>2$, we assume that $kL>N$ ($L\geq l>0$ if $N=2$ for some strictly positive $l$). We
keep the notations of the previous sections.

Let $\chi\in C_{0}^{\infty}((-1,1))$, $\chi=1$ close to $0$
and for every $i$, set
$\chi_{i}(x):=\chi(\frac{|x-o_{i}|}{r_{i}}-1)$ and
$\chi_{i,\lambda}(x)=\chi(\lambda^{\alpha}(\frac{|x-o_{i}|}{r_{i}}-1))$
($\chi_{i,\lambda}$ vanishes outside a neighborhood of size
$\lambda^{-\alpha}$ of the ball $\Theta_{i}$). Set
$u=e^{it\Delta_{D}}\psi(\frac{-\Delta_{D}}{\lambda^{2}})u_{0}$ and
$v_{i}=\chi_{i}u$. Then $v_{i}$ satisfies
\begin{equation}\label{LS2}
(i\partial_{t}+\Delta_{D})v_{i}=[\Delta_{D},\chi_{i}]u,\quad
v_{i}|_{t=0}=\chi_{i}\psi(\frac{-\Delta_{D}}{\lambda^{2}})u_{0}.
\end{equation}
We introduce the Dirichlet Laplacian $\Delta_{D}^{i}$ acting
on $\Omega_{i}:=\mathbb{R}^{3}\setminus\Theta_{i}$ and continue to denote $\Delta_{D}$
the Dirichlet Laplacian outside the obstacle $\Theta$. Writing
Duhamel's formula we get
\begin{equation}\label{duhamel}
v_{i}(t)=e^{it\Delta_{D}^{i}}\chi_{i}\psi(\frac{-\Delta_{D}}{\lambda^{2}})u_{0}-i
\int_{0}^{t}e^{it\Delta_{D}^{i}}\psi(\frac{-\Delta_{D}}{\lambda^{2}})e^{-it\tau\Delta_{D}}
[\Delta_{D},\chi_{i}](u(\tau))d\tau.
\end{equation}
The solution $v_{i}$ of \eqref{LS2} satisfies
\begin{equation}\label{trap}
\|\chi_{i,\lambda}v_{i}\|_{L_{T}^{2}H_{D}^{\frac{1}{2}}(\Omega)}\leq\|\chi_{i,\lambda}e^{it\Delta_{D}^{i}}
\chi_{i}\psi(\frac{-\Delta_{D}^{i}}{\lambda^{2}})u_{0}\|_
{L_{T}^{2}H_{D}^{\frac{1}{2}}(\Omega)}+
\end{equation}
\[
\|\chi_{i,\lambda}\int_{0}^{t}e^{it\Delta_{D}^{i}}\psi(\frac{-\Delta_{D}}{\lambda^{2}})e^{-it\tau\Delta_{D}}
[\Delta_{D},\chi_{i}](u(\tau))d\tau\|_{L_{T}^{2}H_{D}^{\frac{1}{2}}(\Omega)}.
\]
From \cite[Thm.4.2]{Bu1} we know that the operator 
\begin{equation}\label{opai}
A_{i}^{*}f_{i}:=\int_{0}^{T}\psi(\frac{-\Delta_{D}}{\lambda^{2}})e^{-i\tau\Delta_{D}}\chi_{i}
f_{i}(\tau)d\tau
\end{equation}
is bounded from
${L_{T}^{2}H_{D}^{-\frac{1}{2}}(\Omega)}$ to
$H_{D}^{-\epsilon}(\Omega)$. Notice that in \eqref{opai} we can introduce a cutoff function $\tilde{\chi}_{i}\in C^{\infty}_{0}$ equal to $1$ on the support of $\chi_{i}$ without changing the integral modulo smoothing terms.
We need the next lemma:
\begin{lemma}\label{lempsi}
Let $\tilde{\psi}\in
C_{0}^{\infty}(\mathbb{R}_{+}^{*})$ be a smooth function such that $\tilde{\psi}(\tau)=1$ if $\psi(\tau)=1$ and such that $\sum_{k\geq 0}\tilde{\psi}(2^{-2k}\tau)=1$. Let $\check{\chi}_{i}\in C^{\infty}_{0}$ be equal to $1$ on the support of $\chi_{i}$. Then for $f\in L^{2}(\Omega)$
\begin{equation}
\tilde{\psi}(\frac{-\Delta^{i}_{D}}{\lambda^{2}})\tilde{\chi}_{i}\psi(\frac{-\Delta_{D}}{\lambda^{2}})(\chi_{i}f)=\tilde{\chi}_{i}\psi(\frac{-\Delta_{D}}{\lambda^{2}})(\chi_{i} f)
+O_{L^{2}(\Omega)}(\lambda^{-\infty})\chi_{i}f.
\end{equation}
\end{lemma}
We postpone the proof of Lemma \ref{lempsi} for the end of this section.

\emph{End of the proof of Theorem \ref{thm2}:}
We introduce the operator 
\[
A_{i,\lambda}u_{0}(t,x):=\chi_{i,\lambda}e^{it\Delta_{D}^{i}}\tilde{\psi}(\frac{-\Delta_{D}^{i}}{\lambda^{2}})u_{0},
\]
which 
is continuous from $L^{2}(\Omega_{i})$ to
$L_{T}^{2}H_{D}^{\frac{1}{2}}(\Omega_{i})$ with the norm bounded
by $\lambda^{-\alpha/4}$. Indeed, since $\chi_{i,\lambda}$ vanishes outside a small neighborhood of
$\Theta_{i}$ we can apply Theorem \ref{thm1} in $\mathbb{R}^{3}\setminus\Theta_{i}$. 
If we take $f_{i}=[\Delta_{D},\chi_{i}]u$, then in view of Lemma \ref{lempsi} the last term in the right hand
side of \eqref{trap} writes $A_{i,\lambda}A_{i}^{*}f_{i}+O_{L^{2}(\Omega)}(\lambda^{-\infty})$.

If $\check{\chi}_{i}\in C^{\infty}_{0}$ equals $1$ on the support of $\chi_{i}$ we can estimate
\begin{equation}
\|[\Delta_{D},\chi_{i}]u\|_{L_{T}^{2}H_{D}^{-\frac{1}{2}}(\Omega)}\lesssim\|[\Delta_{D},\chi_{i}]\check{\chi_{i}}u\|_
{L_{T}^{2}H_{D}^{-\frac{1}{2}}(\Omega)}\lesssim
\end{equation}
\[
\|\check{\chi_{i}}e^{it\Delta_{D}}\psi(\frac{-\Delta_{D}}{\lambda^{2}})u_{0}\|_
{L_{T}^{2}H_{D}^{\frac{1}{2}}(\Omega)}\lesssim\|\psi(\frac{-\Delta_{D}}{\lambda^{2}})u_{0}\|_{H_{D}^{\epsilon}(\Omega)},
\]
where in the last inequality we applied \cite[Thm.1.7]{Bu1}.
Hence \eqref{trap} becomes
\begin{equation}\label{tra}
\|\chi_{i,\lambda}u\|_{L_{T}^{2}H_{D}^{\frac{1}{2}}(\Omega)}\lesssim\lambda^{-\frac{\alpha}{4}}\|\chi_{i,\lambda}\psi(\frac{-\Delta_{D}}{\lambda^{2}})
u_{0}\|_{L_{T}^{2}(\Omega)}+\lambda^{-\frac{\alpha}{4}+2\epsilon}\|\chi_{i,\lambda}
\psi(\frac{-\Delta_{D}}{\lambda^{2}})u_{0}\|_{L^{2}(\Omega)}.
\end{equation}
Set $\chi_{\lambda}:=\sum_{i=1}^{N}\chi_{i,\lambda}$. Since
$\{\chi_{i,\lambda}\}$ have disjoint supports, \eqref{tra} remains
valid for $\chi_{i,\lambda}$ replaced by $\chi_{\lambda}$. 
We have thus obtained a smoothing effect with a gain
$\alpha/4-2\epsilon$ and by interpolation with the energy estimate we have
\begin{equation}\label{inter}
\|\chi_{\lambda}u\|_{L^{p}_{T}L^{q}(\Omega)}\lesssim\lambda^{\frac{1}{p}(1-\frac{\alpha}{2}+\frac{\epsilon}{4})}\|\Psi(-\frac{\Delta_{D}}{\lambda^{2}})u_{0}\|_{L^{2}(\Omega)}.
\end{equation}
Away from $\partial\Theta$ we can
use the same arguments as in Section \ref{away} and find
\begin{equation}\label{exter}
\|(1-\chi_{\lambda})u\|_{L^{p}_{T}L^{q}(\Omega)}\lesssim\lambda^{\alpha}\|\Psi(-\frac{\Delta_{D}}{\lambda^{2}})u_{0}\|_{L^{2}(\Omega)}.
\end{equation}
Let $p=2+\epsilon$ (for $\epsilon>0$ sufficiently small) and take $\frac{1}{2+\epsilon}(1-\frac{\alpha}{2}+\frac{\epsilon}{4})=\alpha$.

\begin{proof}\emph{(of Lemma \ref{lempsi}):}
We fix $\lambda=2^{k_{0}}$, $k_{0}\geq 1$. We write
\begin{equation}\label{psidesc}
\tilde{\chi}_{i}\psi(-2^{-2k_{0}}\Delta_{D})=\sum_{k\geq 0}\tilde{\psi}(-2^{-2k}\Delta^{i}_{D})
\tilde{\chi}_{i}\psi(-2^{-2k_{0}}\Delta_{D})=
\end{equation}
\[
=\tilde{\psi}(-2^{-2k_{0}}\Delta^{i}_{D})
\tilde{\chi}_{i}\psi(-2^{-2k_{0}}\Delta_{D})+\sum_{k\neq k_{0}}\tilde{\psi}(-2^{-2k_{0}}(2^{-2(k-k_{0})}\Delta^{i}_{D}))
\tilde{\chi}_{i}\psi(-2^{-2k_{0}}\Delta_{D}).
\]
Let $\tilde{B}\subset\mathbb{R}^{3}$ be a ball of sufficiently large radius such that $\cup_{i=1}^{N}\text{supp}\tilde{\chi}_{i}\subset\subset B$ and $\tilde{B}_{i}\subset\mathbb{R}^{3}$ be balls such that $\text{supp}\tilde{\chi}_{i}\subset \tilde{B}_{i}$. Suppose that $\tilde{\chi}_{i}$ are suitably chosen such that the distances between any two such balls $\tilde{B}_{i}$ be strictly positive (this is always possible, eventually shrinking the supports of $\chi_{i}$). If we denote $\tilde{\Delta}_{D}$, resp. $\tilde{\Delta}^{i}_{D}$ the Laplace operators in the bounded domains $\tilde{B}\cap\Omega$ and resp. $\tilde{B}_{i}\cap\Omega$, we notice that 
\begin{equation}\label{egallapladi}
\Delta_{D}(\chi_{i}f)=\tilde{\Delta}_{D}(\chi_{i}f)=\Delta^{i}_{D}(\chi_{i}f)=\tilde{\Delta}^{i}_{D}(\chi_{i}f),\quad \forall f\in L^{2}(\Omega).
\end{equation}
Let $\phi_{n}$, $n\geq 1$, resp. $\phi^{i}_{m}$, $m\geq 1$ denote an orthonormal system of eigenfunctions of $\tilde{\Delta}_{D}$, resp. $\tilde{\Delta}^{i}_{D}$ associated to the eigenvalues $\lambda^{2}_{n}$, resp. $\lambda^{i2}_{m}$, i.e. such that $-\tilde{\Delta}_{D}\phi_{n}=\lambda^{2}_{n}\phi_{n}$,
$-\tilde{\Delta}^{i}_{D}\phi^{i}_{m}=\lambda^{i2}_{m}\phi^{i}_{m}$ for $n,m\geq 1$. Consequently if $f_{i}\in L^{2}(\Omega)$ is supported in $\tilde{B}_{i}\cap\Omega$ we have
\[
\psi(-2^{-2k_{0}}\Delta_{D})f_{i}=\sum_{n\geq 1}<f_{i},\phi_{n}>\psi(2^{-2k_{0}}\lambda^{2}_{n})\phi_{n},
\]
and if $g_{i}\in L^{2}(\Omega)$ is supported in $\tilde{B}_{i}\cap\Omega$ and we set $A_{k}=2^{-2(k-k_{0})}$, $k\geq 0$ than
\[
\tilde{\psi}(-2^{-2k_{0}}A_{k}\Delta^{i}_{D})g_{i}=\sum_{m\geq 1}<g_{i},\phi^{i}_{m}>\tilde{\psi}(2^{-2k_{0}}A_{k}\lambda^{i2}_{m})\phi^{i}_{m},
\]
If the support of $\tilde{\psi}$ is sufficiently large we show that the contribution of the sum in the second line of \eqref{psidesc} for $k\neq k_{0}$ is $O_{L^{2}(\Omega)}(\lambda^{-\infty})$. We consider separately the cases $k>k_{0}$, resp. $k<k_{0}$. Let first $k>k_{0}$: we distinguish two case, according to $2^{\epsilon k_{0}}\lesssim 2^{k}<2^{k_{0}}$ for some $\epsilon\geq 1/4$ or $2^{k}\lesssim 2^{k_{0}/4}$.
\begin{itemize}
\item Let $2^{k}< 2^{k_{0}/4}$ and set $A_{k}=2^{2(k_{0}-k)}$, then $A_{k}> 2^{k_{0}}=\lambda\simeq \lambda _{n}$.
\begin{equation}\label{psinm}
\tilde{\psi}(-2^{-2k_{0}}A_{k}\Delta^{i}_{D})\tilde{\chi}_{i}\psi(-2^{-2k_{0}}\Delta_{D})\chi_{i}f=
\end{equation}
\[
\sum_{n,m\geq 1}\tilde{\psi}(2^{-2k_{0}}A_{k}\lambda^{i2}_{m})\psi(2^{-2k_{0}}\lambda^{2}_{n})<\chi_{i}f,\phi_{n}><\tilde{\chi}_{i}\phi_{n},\phi^{i}_{m}>\phi^{i}_{m},
\]
\begin{equation}\label{pronm}
<\tilde{\chi}_{i}\phi_{n},\phi^{i}_{m}>=\frac{1}{\lambda_{n}^{2}}<-\tilde{\chi}_{i}\tilde{\Delta}_{D}\phi_{n},\phi^{i}_{m}>=
\frac{1}{\lambda_{n}^{2}}<-\tilde{\Delta}_{D}(\tilde{\chi}_{i}\phi_{n})+[\tilde{\Delta}_{D},\tilde{\chi}_{i}]\phi_{n},\phi^{i}_{m}>.
\end{equation}
Since $-\tilde{\Delta}_{D}(\tilde{\chi}_{i}\phi_{n})=-\tilde{\Delta}^{i}_{D}(\tilde{\chi}_{i}\phi_{n})$ is self-adjoint, the first term in the last line writes
\[
\frac{1}{\lambda^{2}_{n}}<\tilde{\chi}_{i}\phi_{n},-\tilde{\Delta}^{i}_{D}\phi^{i}_{m}>=\frac{\lambda^{i2}_{m}}{\lambda^{2}_{n}}<\tilde{\chi}_{i}\phi_{n},\phi^{i}_{m}>,
\]
and hence its contribution in the sum in \eqref{psinm} is
\begin{equation}\label{estnm}
\frac{\lambda^{i2}_{m}}{\lambda^{2}_{n}}\tilde{\psi}(2^{-2k_{0}}A_{k}\lambda^{i2}_{m})\psi(2^{-2k_{0}}\lambda^{2}_{n})<\tilde{\chi}_{i}\phi_{n},\phi^{i}_{m}>.
\end{equation}
Since $\psi$, $\tilde{\psi}$ are compactly supported away from $0$, the only nontrivial contribution in the sum \eqref{psinm} comes from indices $n$, $m$ such that $\lambda^{2}_{n}\simeq A_{k}\lambda^{i2}_{m}\simeq 2^{2k_{0}}$ and from \eqref{estnm} this will be $O_{L^{2}(\Omega)}(A^{-1}_{k})=O_{L^{2}(\Omega)}(1/\lambda_{n})$ which follows from the assumption $2^{\epsilon k_{0}}\lesssim 2^{k}$ for some $\epsilon\geq 1/4$. We estimate the last term in the right side of \eqref{pronm} 
\[
\frac{1}{\lambda^{2}_{n}}<[\tilde{\Delta}_{D},\tilde{\chi}_{i}]\phi_{n},\phi^{i}_{m}>=O_{L^{2}(\Omega)}(\lambda^{-1}_{n})=O_{L^{2}(\Omega)}(2^{-k_{0}}),
\]
since on the support of $\psi$, $\lambda_{n}\simeq 2^{k_{0}}$.
By iterating these arguments $M\geq 1$ times we deduce that the contribution of the sum \eqref{psidesc} is $O_{L^{2}(\Omega)}(2^{-Mk_{0}})$ for every $M\geq 1$.

\item Let now $2^{\epsilon k_{0}}\lesssim 2^{k}<2^{k_{0}}$, $\epsilon\geq 1/4$: in this case a simple integration by part is useless since the "error" is a multiple of the number of integrations.
If the support of $\tilde{\chi}_{i}$ is sufficiently small then \eqref{egallapladi} holds and we have
\begin{equation}\label{phinphimchii}
<\tilde{\chi}_{i}\phi_{n},\phi^{i}_{m}>=\int (\tilde{\chi}_{i}-1)\phi_{n}\overline{\phi}^{i}_{m}dx+\delta_{n=m},
\end{equation}
where $\delta_{n=m}$ is the Dirac distribution. In the sum \eqref{psinm} we see that the contribution from $n=m$ is zero, since the support of $\tilde{\psi}(A_{k}.)$ and $\psi(.)$ are disjoint. For first term in the right hand side of \eqref{phinphimchii} we use an argument of N.Burq, P.G\'erard and N.Tzvetkov \cite[Lemma 2.6]{bgtbilin}.
Let $\kappa\in\mathcal{S}(\mathbb{R})$ be a rapidly decreasing function such that $\kappa(0)=1$. 
 From a result of Sogge \cite[Chp.5.1]{so93} we can write, on the support of $\tilde{\chi}$ where \eqref{egallapladi} holds
\[
\kappa(\sqrt{-\tilde{\Delta}_{D}}-\lambda)f(x)=\lambda^{1/2}\int e^{i\lambda\varphi(x,y)}a(x,y,\lambda)f(y)dy+R_{\lambda}f(x),
\]
where $a(x,y,\lambda)\in C^{\infty}_{0}$ is an asymptotic assumption in $1/\lambda$ and $-\varphi(x,y)$ is the geodesic distance between $x$ and $y$, and where
\[
\forall p,s\in \mathbb{N}:\quad \|R_{\lambda}f\|_{H^{s}(\text{supp}\tilde{\chi}_{i})}\leq C_{p,s}\lambda^{-p}\|f\|_{L^{2}}.
\]  
We use \eqref{egallapladi} and $\kappa(\sqrt{-\Delta_{D}}-\lambda_{n})\phi_{n}=\phi_{n}$, $\kappa(\sqrt{-\Delta^{i}_{D}}-\lambda^{i}_{m})\phi^{i}_{m}=\phi^{i}_{m}$ in order to write $<(\tilde{\chi}_{i}-1)\phi_{n},\phi^{i}_{m}>$ as an integral (modulo a remaining term, small) 
\[
\int_{x,y,z}e^{i\lambda^{1/2}_{n}\Phi(x,y,z)}a(x,y,\lambda_{n})\bar{a}(x,z,\lambda^{i}_{m})\phi_{n}(y)\bar{\phi}^{i}_{m}(z)dy dzdx,
\]
with $\Phi(x,y,z)=\varphi(x,y)+\sqrt{\lambda^{i}_{m}/\lambda_{n}}\varphi(x,z)$. Since $|\nabla_{x}\varphi|$ is uniformly bounded from below and bounded from above together with all its derivatives, we obtain that there exist $c>0$, $C_{\beta}>0$ such that $|\nabla_{x}\Phi|\geq c$ and $\partial^{\beta}\Phi\leq C_{\beta}$. It remains to perform integrations by parts in the $x$ variable as many times as we want, each such integration providing a gain of a power $\lambda_{n}^{-1/2}$. 
\end{itemize}
For $k>k_{0}$ and $2^{k_{0}}\lesssim 2^{k/4}$, we write
\begin{equation}\label{pronml}
<\tilde{\chi}_{i}\phi_{n},\phi^{i}_{m}>=\frac{1}{\lambda^{i2}_{m}}<-\tilde{\chi}_{i}\phi_{n},\tilde{\Delta}^{i}_{D}\phi^{i}_{m}>=
\frac{1}{\lambda^{i2}_{m}}<-\tilde{\Delta}_{D}(\tilde{\chi}_{i}\phi_{n})+[\tilde{\Delta}_{D},\tilde{\chi}_{i}]\phi_{n},\phi^{i}_{m}>,
\end{equation}
in which case, using the spectral localizations $\psi$, $\tilde{\psi}$, we gain a factor $A_{k}$ from the first term in \eqref{pronml} and a factor $A_{k}(\lambda^{i}_{m})^{-1}$ from the second one; iterating the argument as many times as we want we obtain contribution $O_{L^{2}(\Omega)}(A^{M}_{k})$. In the last case $2^{\epsilon k}\lesssim 2^{k_{0}}<2^{k}$, $\epsilon\geq 1/4$ we use again the arguments of \cite[Lemma 2.6]{bgtbilin}.
The proof is complete.
\end{proof}

\section{Appendix}
\subsection{Hankel functions}
We will start by recalling some properties of Hankel functions
that will be useful in determining the behavior of the solution to
\eqref{2d} (see \cite{Ab}). The Hankel function of order $\nu$ is
defined by
\begin{equation}
H_{\nu}(z)=\int_{-\infty}^{+\infty-i\pi} e^{z\sinh t-\nu t}dt.
\end{equation}
For all values of $\nu$, $\{H_{\nu}(z),\bar{H}_{\nu}(z)\}$
form a pair of linearly independent solutions of the Bessel's
equation $z^{2}y''+zy'+(z^{2}-\mu^{2})y=0$.
\begin{prop}\label{propasyexp}
We have the following asymptotic expansions (see \cite[Chp.9]{Ab}):
\begin{enumerate}

\item If $\nu$ is fixed, bounded and
$z=r\lambda\gg\nu>\frac{1}{2}$, $\frac{\nu}{z}\ll 1$ then
\begin{equation}\label{Ai}
H_{\nu}(z)\simeq\sqrt{\frac{2}{\pi z}}e^{i(\lambda
r-\frac{\pi\nu}{2}-\frac{\pi}{4})}(1+O(\lambda^{-1})),\quad \bar{H}_{\nu}(z)\simeq\sqrt{\frac{2}{\pi z}}e^{-i(\lambda
r-\frac{\pi\nu}{2}-\frac{\pi}{4})}(1+O(\lambda^{-1}));
\end{equation}

\item If $z=r\lambda>\nu\gg 1$ and $\frac{\nu}{z}\in[\epsilon_{0},1-\epsilon_{1}]$ for some small, fixed $\epsilon_{0}>0$, $\epsilon_{1}>0$, then writing $\frac{z}{\nu}=\frac{1}{\cos\beta}$ we have
\begin{equation}\label{Aiia}
H_{\nu}(z)=H_{\nu}(\frac{\nu}{\cos\beta})\simeq \sqrt{\frac{2}{\pi\nu
\tan\beta}}e^{i\lambda(\tan\beta-
\beta)+\frac{i\pi}{4}}(1+O(\nu^{-1})),
\end{equation}
\begin{equation}\label{Aiib}
\bar{H}_{\nu}(z)=\bar{H}_{\nu}(\frac{\nu}{\cos\beta})\simeq \sqrt{\frac{2}{\pi\nu
\tan\beta}}e^{i\lambda(\tan\beta-
\beta)+\frac{i\pi}{4}}(1+O(\nu^{-1}));
\end{equation}

\item If $z,\nu\gg 1$ are nearly equal we have the following formulas
\begin{enumerate}

\item If $z-\nu=\tau\nu^{1/3}$ with fixed $\tau$, bounded, then
\begin{equation}\label{Aiiiaj}
J_{\nu}(\nu+\tau\nu^{\frac{1}{3}})\simeq\frac{\sqrt[3]{2}}{\sqrt[3]{\nu}}Ai(-\sqrt[3]{2}\tau)(1+
O(\nu^{-1})),
\end{equation}
\begin{equation}\label{Aiiiay}
Y_{\nu}(\nu+\tau\nu^{\frac{1}{3}})\simeq-\frac{\sqrt[3]{2}}{\sqrt[3]{\nu}}(Bi(-\sqrt[3]{2}\tau)(1+
O(\nu^{-1})),
\end{equation}
where for $|\tau|$ large and $\xi=\frac{2}{3}\tau^{\frac{3}{2}}$ we have
\begin{equation}\label{aibim}
Ai(-\tau)\simeq\frac{1}{\pi^{\frac{1}{2}}\tau^{\frac{1}{4}}}\sin(\xi+\frac{\pi}{4})(1+O(\xi^{-1})),
\quad
Bi(-\tau)\simeq\frac{1}{\pi^{\frac{1}{2}}\tau^{\frac{1}{4}}}\cos(\xi+\frac{\pi}{4})(1+O(\xi^{-1}));
\end{equation}

\item If $z=\nu$, then
\begin{equation}\label{Aiiibj}
J_{\nu}(\nu)\simeq\frac{2^{1/3}}{3^{2/3}\Gamma(2/3)}\nu^{-1/3}(1+O(\nu^{-1})),
\end{equation}
\begin{equation}\label{Aiiiby}
 Y_{\nu}(\nu)\simeq-\frac{2^{1/3}}{3^{1/6}\Gamma(2/3)}\nu^{-1/3}(1+O(\nu^{-1}));
\end{equation}

\item If $|\nu-r\lambda|\leq C_{0}|r\lambda|$ then if $\nu z=r\lambda$ we have for $z<1$ (resp. $z>1$)
\begin{equation}\label{Aiiicj}
J_{\nu}(\nu z)\simeq(\frac{4\zeta}{1-z^{2}})^{1/4}\Big(\frac{Ai(\nu^{2/3}\zeta)}{\nu^{1/3}}+\frac{\exp(\frac{2}{3}\nu\zeta^{3/2})}{1+\nu^{1/6}|\zeta|^{1/4}}O(1/\nu^{4/3})\Big),
\end{equation}
\begin{equation}\label{Aiiicy}
Y_{\nu}(\nu z)\simeq-(\frac{4\zeta}{1-z^{2}})^{1/4}\Big(\frac{Bi(\nu^{2/3}\zeta)}{\nu^{1/3}}+\frac{\exp(|\Re(\frac{2}{3}\nu\zeta^{3/2})|)}{1+\nu^{1/6}|\zeta|^{1/4}}O(1/\nu^{4/3})\Big),
\end{equation}
where the function $\zeta$ is defined by
\begin{equation}\label{zetap}
\frac{2}{3}\zeta^{3/2}=\int_{z}^{1}\frac{\sqrt{1-t^{2}}}{t}dt=\ln[(1+\sqrt{1-z^{2}})/z]-\sqrt{1-z^{2}},\quad z\leq 1,
\end{equation}
\begin{equation}\label{zetam}
\frac{2}{3}(-\zeta)^{3/2}=\int_{1}^{z}\frac{\sqrt{t^{2}-1}}{t}dt=\sqrt{z^{2}-1}-\arccos{(1/z)},\quad z\geq 1,
\end{equation}
and where for $\tau$ large and $\xi=\frac{2}{3}\tau^{3/2}$ we have
\begin{equation}\label{aibip}
Ai(\tau)=\frac{1}{2\pi^{\frac{1}{2}}\tau^{\frac{1}{4}}}e^{-\xi}(1+O(\xi^{-1})),
\quad
Bi(\tau)=\frac{1}{2\pi^{\frac{1}{2}}\tau^{\frac{1}{4}}}e^{\xi}(1+O(\xi^{-1})).
\end{equation}
Taking $\tau=\nu^{2/3}\zeta$, we compute $Ai(\nu^{2/3}\zeta)$ using \eqref{aibip}, \eqref{aibim} with $\xi=\frac{2}{3}\zeta^{3/2}\nu$.
\end{enumerate}
\end{enumerate}
Here $J_{\nu}$ and $Y_{\nu}$ are the Bessel functions of the first
kind and $H_{\nu}(z)=J_{\nu}(z)+iY_{\nu}(z)$.
\end{prop}
\begin{rmq}
We remark that $\zeta$ defined in \eqref{zetap}, \eqref{zetam} is analytic in $z$, even at $z=1$ and $d\zeta/dz<0$ there; also, at $z=1$, $\zeta=0$ and $(1-z^{2})^{-1}\zeta=2^{-2/3}$. The formulas \eqref{Aiiicj}, \eqref{Aiiicy} are among the deepest and most important results in the theory of Bessel functions. These results do not appear in the treatise of Watson, having been established after Watson's second edition was published. See Olver \cite{ol74} and \cite{Ab}. In the paper we use a simpler form of these asymptotic expansions for which we give an idea of the proof inspired from \cite{chester}.
\end{rmq}
\begin{prop}\label{prophank}
For $z>1$ close to $1$ and $\nu\gg 1$ large enough we have
\begin{equation}\label{jutil}
J_{\nu}(\nu z)\simeq \frac{\sqrt{2}}{\nu^{1/3}}(\frac{\zeta}{z^{2}-1})^{1/4}Ai(-\nu^{2/3}\zeta),
\end{equation}
\begin{equation}\label{yutil}
Y_{\nu}(\nu z)\simeq -\frac{\sqrt{2}}{\nu^{1/3}}(\frac{\zeta}{z^{2}-1})^{1/4}Bi(-\nu^{2/3}\zeta),
\end{equation}
uniformly in $z$, where $\frac{2}{3}\zeta^{3/2}=\sqrt{z^{2}-1}-\arccos(1/z)$ (see \cite{ol74}).
\end{prop}
\begin{proof}
With a suitable choice of contour we have 
\begin{equation}\label{jjj}
J_{\nu}(\nu z)=\frac{1}{2\pi}\int e^{i\nu\phi(t,z)}dt,\quad \phi(t,z)=z\sin t-t.
\end{equation}
For $\nu\gg 1$ and $z>1$ the saddle points are real, $\tilde{t}=\arccos(1/z)$ and the critical value is $\phi(t(z),z)=\sqrt{z^{2}-1}-\arccos(1/z)$. Now if $\zeta(z)$ is defined in terms of the exponent in the Debye expansion, it is analytic in $z$ for $z$ near $1$ and we have
\begin{equation}
Ai(-\nu^{2/3}\zeta)=\frac{1}{2\pi}\int e^{-is\nu^{2/3}\zeta+is^{3}/3}ds=\frac{\nu^{1/3}}{2\pi}\int e^{i\nu(-t\zeta+t^{3}/3)}dt
\end{equation}
with critical points $t^{2}=\zeta$, thus we get \eqref{jutil} applying the stationary phase. We obtain the Debye approximations
\begin{equation}\label{jfaible}
J_{\nu}(\nu z)\simeq \Big(\frac{2}{\pi\nu\sqrt{z^{2}-1}}\Big)^{1/2}\cos[\nu\sqrt{z^{2}-1}-\nu\arccos(1/z)-\pi/4]
\end{equation}
by replacing the Airy function by its asymptotic expansion, thus the proper condition for its validity is $\nu^{2/3}\zeta\gg 1$. For small $\nu^{2/3}\zeta$ we are in the regime $\nu(z-1)=\tau\nu^{1/3}$ for which we have the estimations \eqref{Aiiiaj}, \eqref{Aiiiay}. An extension of this result giving \eqref{Aiiicj}, \eqref{Aiiicy} uses a result of Chester, Friedman and Ursell who showed that a similar reduction is possible whenever two saddle points coalesce: if $\partial_{t}\phi(\tilde{t},1)=0$ and $\partial^{2}_{tt}\phi(\tilde{t},1)=0$ but $\partial^{3}_{ttt}\phi(\tilde{t},1)\neq 0$, then an integral of the form
\eqref{jjj}  has a uniform asymptotic expansion in terms of the Airy function and its derivative. Their method was to make a change of variables so that 
\begin{equation}
\phi(t,z)=\zeta\tau-\tau^{3}/3,
\end{equation}
which holds exactly and uniformly; it is not merely an approximation for $z$ near $1$.
\end{proof}


\nocite{*}
\begin{thebibliography}{abramow}

\bibitem{Ab}
M.~Abramowitz, I.~Stegun,
{\em Handbook of {M}athematical {F}unctions}, {D}over {P}ublications, {N}ew {Y}ork, 1965


\bibitem{Be}
J.~Bergh, J.~Löfstrom,
{\em Interpolation {S}paces}, {S}pringer-{V}erlag, 1976

\bibitem{BeKl}
M.~Ben Artzi, S.~Kleinerman,
{\em Decay and regularity for the {S}chrödinger equation},
{J}.{A}nal.{M}ath., 58:25-37, 1992

\bibitem{Bourg}
J.~Bourgain,
{\em Global {S}olutions of {N}onlinear {S}chrödinger {E}quations},
{C}olloquium {P}ublications, {A}merican {M}athematical {S}ociety, 46, 1999

\bibitem{Bu1}
 N.~Burq,
{\em Smoothing effect for {S}chrödinger boundary value problems}
{D}uke {M}ath.{J}ournal, no.2, vol.123, 2004
 
\bibitem{Bu2}
N.~Burq,
{\em Semi-classical {E}stimates for the {R}esolvent in {N}ontrapping {G}eometries},
{I}nt.{M}ath.{R}es.Not. 5, 2002

\bibitem{bgt03}
N.~Burq, P.~Gérard, N.~Tzvetkov,
{\em On nonlinear {S}chrödinger equations in exterior domains },
{A}nn.{I}.{H}.{P}oincaré, 21:295-318, 2004

\bibitem{bgt04}
N.~Burq, P.~Gérard, N.~Tzvetkov,
{\em Strichartz inequalities and the nonlinear {S}chrödinger equation on compact manifolds},
{A}merican {J}.{M}ath., 126: 569-605, 2004

\bibitem{bgtbilin}
N.Burq, P.G\'erard, N.Tzvetkov, Bilinear eigenfunctions estimates and the nonlinear {S}chr\"odinger equation on surfaces, {Invent.Math.} 150: 187-223, 2005

\bibitem{Ca}
T.~Cazenave,
{\em An introduction to nonlinear {S}chrödinger equations},
{T}extos de {M}étodos {M}atematicos", 26, 19

\bibitem{chester}
C.~Chester, B.~Friedman, C.~Ursell,
{\em An extension of the method of steepest descent}, {P}roc.{C}ambridge{P}hilos.{S}oc.,54:599-611, 1957

\bibitem{CoSa}
P.~Constantin, J.~C.~Saut,
{\em Local smoothing properties of dispersive equations},
{J}.{A}mer.{M}ath.{S}oc., 1:413-439, 1988

\bibitem{cks95}
W.~Craig, T.~Kappeler, W.~Strauss,
{\em Microlocal dispersive smoothing for the {S}chr\"odinger equation}, {C}omm.{P}ures {A}ppl.{M}ath. 48, 8:769-860, 1995

\bibitem{Do}
S.~I.~Doi,
{\em Smoothing effects of {S}chrödinger evolution groups on {R}iemanian manifolds},
{D}uke {M}ath.{J}., 82:679-706, 1996

\bibitem{Gi}
J.~Ginibre, G.~Velo,
{\em The global {C}auchy problem for the nonlinear {S}chrödinger equation},
{A}nn.{I}.{P}., {A}nal.non.lin.,2:309-327, 1985
 93

\bibitem{Ka}
T.~Kato,
{\em On nonlinear {S}chrödinger equations},
{A}nn.{I}nst.{H}.{P}oincaré {P}his.{T}héor.", 46:113-129, 1987

\bibitem{ol74}
F.~Olver,
{\em Asymptotics and {S}pecial functions}, {A}cademic {P}ress, {N}ew {Y}ork, 1974

\bibitem{Re}
M.~Reed, B.~Simon,
{\em Methods of {M}odern {M}athematical {P}hisics {I}-{I}{V}}, {A}cademic {P}ress, 1975

\bibitem{so93} C.D.Sogge, Fourier integrals in classical analysis, {C}ambridge {U}niv.{P}ress, {C}ambridge and {N}ew {Y}ork, 1993
  
 \bibitem{So}
H.~Smith, C.~Sogge,
{\em On the critical semilinear wave equation outside convex obstacles},
{J}.{A}mer.{M}ath.{S}oc., 8:879-916, 1995

\bibitem{SmSo}
  H.~Smith, C.~Sogge,
{\em On {S}trichartz estimates for the {S}chrödinger operators in compact manifolds with boundary}, arxiv, 2006

\bibitem{stta02}
G.~Staffilani, D.~Tataru,
{\em Strichartz estimates for a {S}chrödinger operator with nonsmooth coefficients},
 {C}omm.{P}art.{D}iff.{E}q.,no. 27, vol.5-6, 1337-1372, 2002

\bibitem{Sul}
C.~Sulem, P.~L.~Sulem,
{\em The {N}onlinear {S}chrödinger {E}quation}, {S}pringer, 1999
  
\bibitem{tay96}
M.~Taylor,
{\em Partial {D}ifferential {E}quations {I},{II},{III}}, {S}pringer, 1996

 \bibitem{Wa}
G.~N.~Watson,
{\em A {T}reatise on the {T}heory of {B}essel {F}unctions}, {C}ambride {U}niv.{P}ress 1944


\end{thebibliography}
\end{document}